\documentclass{amsart}
\usepackage{amssymb}
\usepackage{amsmath}

\usepackage{bbm}
\usepackage{mathrsfs}

\newtheorem{lem}{Lemma}[section]

\newtheorem{dfn}[lem]{Definition}
\newtheorem{pro}[lem]{Proposition}
\newtheorem{thm}[lem]{Theorem}

\newtheorem{cor}[lem]{Corollary}
\newtheorem{rem}[lem]{Remark}

\usepackage{color}

\def\F{\mathcal F}

\def\N{\mathbb N}
\def\ep{\varepsilon}
\def \bt{\beta}
\newcommand\hdim{\dim_{\mathrm H}}

\def\Z{\mathbb Z}

\def\H{\mathcal H}

\def\L{\mathcal L}
\def\mR{\mathcal R}
\def\mQ{\mathcal Q}
\def\hv{\hat{v}}



\begin{document}
\baselineskip 14pt

\title{Uniform Diophantine approximation on the plane for $\beta$-dynamical systems}

\author{Xiaohui Fu$^1$, Junjie Shi$^2$ and Chen Tian$^{3,\ast}$}

\address{$^1$School of Mathematics and Statistics, Huazhong University of Science and Technology, Wuhan 430074, P. R. China}
\email{xhfu@hust.edu.cn}

\address{$^2$Department of Mathematics, Taiyuan University, Taiyuan 030000, P. R. China}
\email{junjie\_shi2025@163.com}

\address{$^3$School of Statistics and Mathematics, Hubei University of Economics, Wu Han, Hubei 430205, P. R. China}
\email{tchen@hbue.edu.cn}

\thanks{$^{\ast}$ Corresponding author.}









\begin{abstract}
In this paper, we investigate the two-dimensional uniform Diophantine approximation in $\beta$-dynamical systems. Let $\beta_i > 1(i=1,2)$ be real numbers, and let $T_{\beta_i}$ denote the $\beta_i$-transformation defined on $[0, 1]$. For each $(x, y) \in[0,1]^2$, we define the asymptotic approximation exponent
$$
v_{\beta_1, \beta_2}(x, y)=\sup \left\{0 \leq v<\infty: \begin{array}{l}
T_{\beta_1}^n x<\beta_1^{-n v} \\
T_{\beta_2}^n y<\beta_2^{-n v}
\end{array} \text { for infinitely many } n \in \mathbb{N}\right\} \text {, }
$$
and the uniform approximation exponent
$$
\hat{v}_{\beta_1, \beta_2}(x, y)=\sup \left\{0 \leq \hat{v}<\infty: \forall~ N \gg 1, \exists 1 \leq n \leq N \text { such that } \begin{array}{l}
T_{\beta_1}^n x < \beta_1^{-N \hat{v}} \\
T_{\beta_2}^n y < \beta_2^{-N \hat{v}}
\end{array}\right\} .
$$
We calculate the Hausdorff dimension of the intersection $$\left\{(x, y) \in[0,1]^2: \hat{v}_{\beta_1, \beta_2}(x, y)=\hat{v} \text { and } v_{\beta_1, \beta_2}(x, y)=v\right\}$$ for any $\hat{v}$ and $v$ satisfying $\log _{\beta_2}{\beta_1}>\frac{\hat{v}}{v}(1+v)$. As a corollary, we establish a definite formula for the Hausdorff dimension of the level set of the uniform approximation exponent.
\end{abstract}
\maketitle
\section{Introduction}
The classical Diophantine approximation aims to quantify  how closely real numbers can be approximated by rational numbers. A qualitative result is provided by the fact that the rational numbers are dense in the set of real numbers. Dirichlet's theorem is a fundamental quantitative result in this context.
\begin{thm}[\cite{DP42}]\label{Dirichlet them}
For any $x \in \mathbb{R}$ and $Q>1$, there exists an integer $q \in  \mathbb{N}$ such that
\begin{equation*}
\|q x\| \leq \frac{1}{Q} \text { and } q<Q .
\end{equation*}
Here $\|\cdot\|$ denotes the distance to the nearest integer, i.e., $\|x\|=\min_{n\in \Z} |x-n|.$
\end{thm}
As an application, we have the following corollary.  It should be noted that this corollary was established before Dirichlet's theorem (see Legendre's book \cite{LE08}).
\begin{cor}[\cite{LE08}]\label{Dirichlet cor}
For any $x \in \mathbb{R}$, there exist infinitely many integers $q$ such that
\begin{equation*}
\|q x\| <\frac{1}{q}.
\end{equation*}
\end{cor}
The above two results characterize a uniform and an asymptotic rate of rational approximation to all real numbers, respectively. In this context, the metric Diophantine approximation aims to improve the approximation rates in the two statements mentioned above on the one hand, and to quantify the size of the intrinsic sets in terms of measures and dimensions on the other hand.

If we denote the rotation mapping on the unit circle by $R_\alpha: x \mapsto x+\alpha$, Dirichlet's theorem and its corollary represent an investigation of the approximation properties of the orbit $\{R_\alpha^n0:n\ge1\}.$ This motivates us to study the asymptotic approximation problems and uniform approximation problems in general dynamical systems. Regarding the approximation rate, it is natural to extend our discussion to general approximating functions. More specifically, given a metric space $(X, d)$, a point $x_0 \in X$, a mapping $T: X \rightarrow X$ and a positive function $\psi$, we are concerned with the size of the set
\begin{equation*}
A\left(T, \psi, x_0\right)=\left\{x \in X: d\left(T^n x, x_0\right)<\psi(n) \text { for infinitely many } n \in \mathbb{N} \right\}
\end{equation*}
and the set
\begin{equation*}
U\left(T, \psi, x_0\right)=\left\{x \in X: \forall N \gg 1, \exists 1 \leq n \leq N \text { such that } d\left(T^n x, x_0\right)<\psi(N)\right\},
\end{equation*}
here and hereafter $N\gg1$ denotes that $N$ is a sufficiently large integer.

When $(X, T)$ is the irrational rotation $([0,1], R_\alpha)$, Fayad \cite{F06} and Kim \cite{K07} studied the Lebesgue measure of the set $A\left(R_\alpha, \psi, x_0\right)$. Bugeaud \cite{B03}, Troubetzkoy and Schmeling \cite{TS03} proved that for any $a > 1,$
\begin{equation*}
\hdim \left\{ x \in [0, 1] : \left\| R_\alpha^n x - x_0 \right\| < \frac{1}{n^a}  \text { for infinitely many } n \in \mathbb{N} \right\} = \frac{1}{a}
,\end{equation*}
where $\hdim$ represents the Hausdorff dimension. Kim and Liao \cite{KL19}, and Kim et al. \cite{KR18} further investigated the size of the set $U\left(R_\alpha, \psi, x_0\right).$

When $(X,T)$ is a system where $T$ is an expanding rational map of degree $\ge 2$ and $X$ the corresponding Julia set, Hill and Velani \cite{HV95} investigated the Hausdorff dimension of the set $A\left(T, \psi, x_0\right).$ When $(X,T)$ is an exponentially mixing system with respect to the probability measure $\mu,$ Kleinbock, Konstantoulas and Richter \cite{KKR23} gave sufficient conditions for $U\left(T, \psi, x_0\right)$ to be of zero or full measure. More results regarding uniform approximation and asymptotic approximation can be found in \cite{GP21,HV97,HV99,KK20,KR22,TZ23}.

In this paper, we study the beta dynamical system $\left([0,1], T_\beta\right)$. For a real number $\beta>1$, we define the $\beta$-transformation $T_\beta:[0,1] \rightarrow[0,1]$ by
\begin{equation}\label{defbt}
T_\beta x=\beta x-\lfloor\beta x\rfloor \quad \text { for } x \in[0,1],
\end{equation}
where $\lfloor\cdot\rfloor$ indicates the integral part of a real number.

In 1957, R\'enyi \cite{R1957} introduced the
map $T_{\bt}$ as a model for expanding real numbers in non-integer bases. Parry \cite{P60} showed that there exists an invariant and ergodic measure under $T_\beta$ that is equivalent to the Lebesgue measure $\mathcal{L}$ defined on $[0,1]$. Then Birkhoff's ergodic theorem yields that for a fixed $ x_0 \in [0, 1] ,$
\begin{equation}\label{eq0}
\liminf_{n \to \infty} \left| T_\beta^n x - x_0 \right| = 0
\end{equation}
for $\mathcal{L}$-almost every $ x \in [0, 1] .$

Taking into account the speed of convergence in \eqref{eq0}, Philipp \cite{P67} proved that for every fixed $x_0 \in [0,1]$,
\begin{equation*}
\mathcal{L}\left(A\left(T_\beta, \varphi, x_0\right)\right)= \begin{cases}0 & \text { if } \sum_{n=1}^{\infty} \varphi(n)<\infty; \\ 1 & \text { if } \sum_{n=1}^{\infty} \varphi(n)=\infty.\end{cases}
\end{equation*}
Shen and Wang \cite{SW13} obtained the Hausdorff dimension of $A\left(T_\beta, \varphi, x_0\right):$
\begin{equation*}
\hdim A\left(T_\beta, \varphi, x_0\right)=\frac{1}{1+\alpha}, \text { \quad where } \alpha=\liminf _{n \to \infty} \frac{-\log _\beta \varphi(n)}{n}.
\end{equation*}
Furthermore, for any $x \in[0,1]$ and any real number $\beta > 1$, Bugeaud and Liao \cite{BL16} introduced the asymptotic approximation exponent, defined as
\begin{equation*}
v_{\beta}(x)=\sup \left\{0 \leq v < \infty:
T_\beta^n x<\left(\beta^n\right)^{-v}
\text { for infinitely many } n \in \mathbb{N} \right\},
\end{equation*}
and the uniform approximation exponent, defined as
\begin{equation*}
\hat{v}_{\beta}(x)=\sup \left\{0 \leq \hat{v} < \infty:
\forall N \gg 1, \exists 1 \leq n \leq N \text { such that } T_\beta^n x<\left(\beta^N\right)^{-\hat{v}} \right\}.
\end{equation*}
They also showed that
\begin{equation*}
\hdim \left\{x \in [0,1]: \hat{v}_\beta(x) \geq \hat{v}\right\}=\hdim\left\{x \in [0,1]: \hat{v}_\beta(x)=\hat{v}\right\}=\left(\frac{1-\hat{v}}{1+\hat{v}}\right)^2
\end{equation*}
for all $\hat{v} \in[0,1]$.

With respect to the two-dimensional case of beta dynamical systems, Wu \cite{WYF23} considered the size of the set
\begin{equation*}W_{\beta_1,\beta_2}(v) =\left\{(x, y) \in[0,1]^2: \begin{array}{l}
T_{\beta_1}^n x<\beta_1^{-n v} \\
T_{\beta_2}^n y<\beta_2^{-n v}
\end{array} \text { for infinitely many } n \in \mathbb{N}\right\},\end{equation*}
where $\beta_2\geq \beta_1>1$ and $v>0.$ More specifically, Wu proved the following result.
\begin{thm}[\cite{WYF23}]\label{wyf}
Let $\beta_1$ and $\beta_2$ be real numbers with $\beta_2\geq \beta_1>1$ and let $v$ be a positive number. Then,
\begin{equation}\label{dim_max_j} \hdim W_{\beta_1,\beta_2}(v) =\left\{\begin{array}{ll}
\frac{2+v-v\log_{\beta_2}\beta_1}{1+v} &\text{if}~\beta_1^{1+v}<\beta_2;\\
\min\left\{\frac{1+\log_{\beta_2}\beta_1}{(1+v)\log_{\beta_2}\beta_1}, \frac{2+v-v\log_{\bt_2}\bt_1}{1+v}\right\} &\text{if}~\beta_1^{1+v}\geq\beta_2.
\end{array}
\right.\end{equation}
\end{thm}

Considering the results discussed above, it is reasonable to further study the intersection of asymptotic approximation problems and uniform approximation problems in beta dynamical systems  in two dimensions. For clarity, we introduce two approximation exponents associated with asymptotic/uniform Diophantine approximation.
\begin{dfn}\label{maxexpnt}
Let $\beta_1$ and $\beta_2$ be real numbers greater than 1. For any $(x,y)\in[0,1]^2,$ define
\begin{equation*}
    v_{\beta_1,\beta_2}(x,y)=\sup \left\{0\leq v<\infty:\begin{array}{l}
T_{\beta_1}^n x<\beta_1^{-n v} \\
T_{\beta_2}^n y<\beta_2^{-n v}
\end{array} \text { for infinitely many } n\in\N \right\},
\end{equation*}
and
\begin{equation*}
    \hat{v}_{\beta_1,\beta_2}(x,y)=\sup \left\{0\leq\hat{v}<\infty: \forall N\gg1,~ \exists ~1\leq n \leq N ~\text{such that  }  \begin{array}{l}
T_{\beta_1}^nx < \beta_1^{-N\hat{v}} \\
T_{\beta_2}^ny < \beta_2^{-N\hat{v}}
\end{array}     \right\}.
\end{equation*}
\end{dfn}
The exponents $v_{\beta_1,\beta_2}(x,y)$ and $\hat{v}_{\beta_1,\beta_2}(x,y)$ are similar to those defined in \cite{BL16}; see also \cite{BA10,BL05} for more information. With this notation, $W_{\beta_1,\beta_2}(v)$ can be represented as $\{(x,y)\in [0,1]^2: v_{\beta_1,\beta_2}(x,y)\geq v \}.$

We will denote by $A_{\beta_1, \beta_2}(v)$ the level set of the asymptotic approximation exponent
\begin{equation*}
A_{\beta_1, \beta_2}(v)=\left\{(x, y) \in[0,1]^2: v_{\beta_1, \beta_2}(x, y)=v\right\},
\end{equation*}
and by $U_{\beta_1, \beta_2}(\hat{v})$ the level set of the uniform approximation exponent
\begin{equation*}
U_{\beta_1, \beta_2}(\hat{v})=\left\{(x, y) \in[0,1]^2: \hat{v}_{\beta_1, \beta_2}(x, y)=\hat{v} \right\}.
\end{equation*}
For general $\beta_1$ and $\beta_2,$ we study the Hausdorff dimension of the intersection  \begin{equation*}E_{\beta_1, \beta_2}(\hat{v},v)=A_{\beta_1, \beta_2}(v)\cap U_{\beta_1, \beta_2}(\hat{v})\end{equation*} in this paper.

For brevity, let
\begin{equation*}A(\hat{v},v)=\frac{\left(1+\log _{\beta_2} \beta_1\right)(v-\hat{v}-\hat{v} v)}{(1+v)(v-\hat{v})}+1-\log _{\beta_2} \beta_1,\quad B(\hat{v},v)=\frac{1+\log _{\beta_2} \beta_1}{\log _{\beta_2} \beta_1} \frac{v-\hat{v}-\hat{v} v}{(1+v)(v-\hat{v})}\end{equation*}
and
\begin{equation*}C(\hat{v},v)=\frac{v-\hat{v}-v \hat{v}\left(1+\left(\log _{\beta_2} \beta_1\right)^{-1}\right)}{(1+v)(v-\hat{v})}+1.
\end{equation*}

\begin{thm}\label{thm3}
Let $\beta_2\ge\beta_1>1$ and $\log _{\beta_2}{\beta_1}>\frac{\hat{v}}{v}(1+v)$. Then the following statements hold:
\begin{itemize}
  \item when $\hat{v}=v=0$, the set $E_{\beta_1, \beta_2}(\hat{v},v)$ has full Lebesgue measure;
   \item when $0<\frac{v}{1+v}<\hat{v}\leq \infty$,
$$\hdim E_{\beta_1, \beta_2}(\hat{v}, v)=0;$$
  \item when $0\leq \hat{v} \leq \frac{v}{1+v}<\infty,$ the Hausdorff dimension of $E_{\beta_1, \beta_2}(\hat{v}, v)$ is given by
\begin{align*} \left\{\begin{array}{ll}
   \min\left\{  A(\hat{v},v) ,~B(\hat{v},v)\right\} &\text{if}~\beta_1^{1+v}> \beta_2; \\
   \min\left\{  A(\hat{v},v) ,~C(\hat{v},v)\right\} &\text{if}~\beta_1^{1+v}\le \beta_2.
   \end{array}
  \right.\end{align*}

\end{itemize}\end{thm}

\begin{rem}
Notably, the condition $\log_{\beta_2}{\beta_1} > \frac{\hat{v}}{v}(1 + v)$ is only invoked for estimating the lower bound of $\hdim E_{\beta_1, \beta_2}(\hat{v}, v).$
\end{rem}

As corollaries of Theorem \ref{thm3}, for general $\beta_1$ and $\beta_2$, we can obtain the Hausdorff dimensions of sets $A_{\beta_1, \beta_2}(v)$ and $U_{\beta_1, \beta_2}(\hat{v})$. In particular, when $\bt_1=\bt_2$, the Hausdorff dimensions of these two sets were established by Tian and Peng \cite{TP04}.
\begin{thm}\label{thm1}
Let $\beta_1$ and $\beta_2$ be real numbers greater than 1. For any $0< v < \infty,$ we have
\begin{equation*}
\hdim A_{\beta_1, \beta_2}(v)
=\hdim W_{\beta_1, \beta_2}(v)= \begin{cases}\frac{2+v-v \log _{\beta_2} \beta_1}{1+v} & \text { if } \beta_1^{1+v}<\beta_2; \\ \min \left\{\frac{1+\log _{\beta_2} \beta_1}{(1+v) \log _{\beta_2} \beta_1}, \frac{2+v-v \log _{\beta_2} \beta_1}{1+v}\right\} & \text { if } \beta_1^{1+v} \geq \beta_2.\end{cases}
\end{equation*}
When $v=0,$ $\hdim A_{\beta_1, \beta_2}(v)
=2;$ when $v=\infty,$  $\hdim A_{\beta_1, \beta_2}(v)
=0.$
\end{thm}

\begin{thm}\label{thm2}
Let $\beta_1$ and $\beta_2$ be real numbers greater than 1.  For any $\hat{v}\in[0,1]$, we have
\begin{align*}
&\hdim \{(x,y)\in [0,1]^2: \hat{v}_{\beta_1,\beta_2}(x,y)=\hat{v} \}\\
=&\hdim \{(x,y)\in [0,1]^2: \hat{v}_{\beta_1,\beta_2}(x,y)\geq \hat{v} \}\\
=&\sup _{0\leq v \leq \infty }\hdim E_{\beta_1, \beta_2}(\hat{v}, v).
\end{align*}
Otherwise, for any $\hat{v}>1$,  both set
$ \{(x,y)\in [0,1]^2:  \hat{v}_{\beta_1,\beta_2}(x,y)=\hat{v} \}$ and set $\{(x,y)\in [0,1]^2: \hat{v}_{\beta_1,\beta_2}(x,y)\geq \hat{v} \} $ are countable.
\end{thm}

\section{Preliminaries}

In this section, we collect some elementary properties of $\beta$-expansions. For further details, we refer to \cite{FW12,P60,R1957}.

For any $\beta>1$, let $T_{\beta}$ be defined as in \eqref{defbt}. Then, any $x\in [0,1]$ can be uniquely represented as
\begin{align}\label{rep1}
  x=&\frac{\ep_1(x,\bt)}{\bt}+\frac{\ep_2(x,\bt)}{\bt^2}+\cdots+\frac{\ep_n(x,\bt)+T^n_\bt(x)}{\bt^{n}} \\
   =&\frac{\ep_1(x,\bt)}{\bt}+\frac{\ep_2(x,\bt)}{\bt^2}+\cdots+\frac{\ep_n(x,\bt)}{\bt^n}+\cdots, \nonumber
\end{align} where $\varepsilon_n(x,\beta)=\lfloor \beta T^{n-1}_{\beta}x \rfloor$ is called the $n$-th digit of  $x.$ The sequence $\varepsilon(x,\beta)=\varepsilon_1(x,\beta)\varepsilon_2(x,\beta)\cdots$ is called the $\beta$-expansion of $x\in[0,1].$

If the sequence $\varepsilon(1,\beta)$ ends with $0^\infty=00\cdots,$ $\beta$ is called a simple Parry number, and we define
\begin{equation*}
\varepsilon^*(1,\beta)=\varepsilon_1(1,\beta)\cdots \varepsilon_n^{-}(1,\beta)\varepsilon_1(1,\beta)\cdots \varepsilon_n^{-}(1,\beta)\cdots =:( \varepsilon_1(1,\beta)\cdots \varepsilon_n^{-}(1,\beta))^\infty,
\end{equation*} where $\varepsilon_n(1,\beta)$ is the last digit of $\varepsilon(1,\beta)$ that is not equal to 0, and $\varepsilon_n^{-}(1,\beta)=\varepsilon_n(1,\beta)-1.$ Otherwise, we define the sequence $\varepsilon^*(1,\beta)$ as $\varepsilon(1,\beta).$ Note that the set of simple Parry numbers is everywhere dense in $(1, \infty),$ see \cite{P60} for a proof.

One can easily verify that for any $x\in[0,1],$ all digits $\varepsilon_n(x,\beta)$ are in $ \mathcal{A}:=\{0,1,\cdots,\lfloor \beta \rfloor \}.$
However, not every sequence in $\mathcal{A}^{\N}$ corresponds to the $\beta$-expansion for some $x\in[0,1].$ A finite or infinite sequence
$\varepsilon_1\varepsilon_2\cdots$ is called  $\beta$-admissible if there exists an $x \in [0, 1]$ such that the $\beta$-expansion of $x$ begins with $\varepsilon_1\varepsilon_2\cdots.$

A lexicographical order $\prec$ on $\mathcal{A}^{\N}$ is defined as follows:
\begin{equation*}
\xi_1\xi_2 \cdots \prec\varepsilon_1\varepsilon_2 \cdots,
\end{equation*}
if there exists $m\geq0$ such that $\xi_1\cdots\xi_m= \varepsilon_1\cdots\varepsilon_m,$ and $\xi_{m+1}<\varepsilon_{m+1}.$ The notation $\xi\preceq \varepsilon$ means that $\xi \prec \varepsilon$ or $\xi=\varepsilon.$ This ordering extends to finite sequences by identifying a finite sequence $\xi_1\cdots\xi_n$ with the infinite sequence $\xi_1\cdots\xi_n00 \cdots.$

The following lemma is a characterization of the admissible sequences.
\begin{lem}[\cite{P60}]\label{Parry}
Let $\beta>1.$
\begin{enumerate}
    \item  A digit sequence $\xi=\xi_1\xi_2\cdots\in \mathcal{A}_\beta^{\N}$ is $\beta$-admissible  if and only if
      \begin{equation*}\sigma^i\xi \prec \varepsilon^*(1,\beta) \quad \text{for all}~ i\geq1,\end{equation*}
      where $\sigma$ is the shift operator such that $\sigma \xi = \xi_2\xi_3\cdots .$
    \item  If $1 <\beta_1 < \beta_2,$ then \begin{equation*}\varepsilon^*(1,\beta_1) \prec \varepsilon^*(1,\beta_2).\end{equation*}
  \end{enumerate}
\end{lem}

For $n\geq 1$, let $ \Sigma_\beta^n$ denote the set of all $\beta$-admissible sequences of length $n,$ i.e.,
\begin{equation*}
\Sigma_\beta^n=\{\omega_1\cdots \omega_n\in \mathcal{A}_\beta^n:  \exists~ x\in [0,1]~ \text{such that}~ \varepsilon_1(x,\beta)=\omega_1,\cdots, \varepsilon_n(x,\beta)=\omega_n \}.
\end{equation*}

The following lemma provides an estimate for the cardinality of the set $\Sigma_\beta^n.$
\begin{lem}[\cite{R1957}]\label{number}
  Let $\beta>1.$ Then for any $n\geq1,$
  \begin{equation*} \beta^n\leq \sharp \Sigma_\beta^n \leq \frac{\beta^{n+1}}{\beta-1},\end{equation*}
  where $\sharp$ denotes the cardinality of a finite set.
\end{lem}

For any $n\ge1$ and any $\omega=\omega_1\cdots \omega_n\in \Sigma_\beta^n,$ the set
\begin{equation*}
I_{n,\beta}(\omega)=\{x \in [0,1]: \varepsilon_1(x,\beta)=\omega_1,\cdots ,\varepsilon_n(x,\beta)=\omega_n\}
\end{equation*} is called an $n$th order basic interval (or an $n$th order cylinder) with respect to the $\beta$. As proved in \cite{FW12}, $I_{n,\beta}(\omega)$ is a left-closed and right-open interval with length at most $\beta^{-n}$. Furthermore, the notion of ``full
cylinder" was introduced in \cite{FW12}. Let $|A|$ denote the diameter of a set $A$.

\begin{dfn}
Let $\omega\in \Sigma_{\beta}^n.$ A basic interval
$I_{n,\beta}(\omega)$ is called full if
\begin{equation*}
|I_{n,\beta}(\omega) |=\frac{1}{\beta^{n}}.
\end{equation*}
Furthermore, the word corresponding to the full cylinder is also said to be full.
\end{dfn}

\begin{lem}[\cite{FW12}, \cite{SW13}]\label{fullcy} Let $\varepsilon=\varepsilon_1\cdots\varepsilon_n \in\Sigma_{\beta}^n. $
\begin{enumerate}
    \item  A basic interval $I_{n,\beta}(\varepsilon)$ is full if and only if, for
any $\beta$-admissible word $\omega=\omega_1\cdots\omega_m\in\Sigma_{\beta}^m,$ the concatenation $\varepsilon\omega=\varepsilon_1\cdots\varepsilon_n\omega_1\cdots\omega_m$ is also
$\beta$-admissible.
    \item If $I_{n,\beta}(\varepsilon)$ is full, then for any
$\omega\in \Sigma_{\beta}^m$ we have
\begin{equation*}|I_{n+m,\beta}(\varepsilon\omega)| = |I_{n,\beta}(\varepsilon)| \cdot |I_{m,\beta}(\omega)| = \beta^{-n}\cdot |I_{m,\beta}(\omega)|.
\end{equation*}
  \end{enumerate}
\end{lem}

The following approximation of the $\beta$-shift is crucial in constructing subsets. For any $N$ with $\varepsilon^*_N(1,\beta)>0,$ let $\beta_N$ be the unique real number satisfying the equation
\begin{equation*}
1=\frac{\varepsilon^*_1(1,\beta)}{z}+\cdots+\frac{\varepsilon^*_N(1,\beta)}{z^N} .
\end{equation*}
Then \begin{equation*}
\varepsilon^*(1,\beta_N)=(\varepsilon^*_1(1,\beta)\cdots (\varepsilon^*_N(1,\beta)-1))^\infty.
\end{equation*}
It can be checked that $\beta_N < \beta$ and the sequence $\{\beta_N\}$ increases and converges to $\beta$ as $N\rightarrow\infty$. By Lemma \ref{Parry} (2), for any $n\geq1$, $\Sigma_{\beta_N}^n \subset  \Sigma_{\beta}^n $.

\begin{lem}[\cite{SW13}]\label{lem1}
For any $\omega \in \Sigma_{\beta_N}^n$, viewed as an element of $\Sigma_{\beta}^n$, the word $\omega0^N$ is full, and
\begin{equation*}
\beta^{-(n+N)}\leq |I_{n,\beta}(\omega)|\leq \beta^{-n}.
\end{equation*}
\end{lem}

We end this section with the following notation. Given two positive real numbers $a$ and $b,$ we write $a\ll b$ or $b\gg a$ if there exists an unspecified positive constant $C$ such that $a\leq Cb.$ We write $a \asymp b$ if $a \ll b$ and $a\gg b.$ For intervals $J_i,J'_i\subset [0,1]$ $(i=1,2),$ let $R=J_1\times J_2,~R'=J'_1\times J'_2$ and
\begin{align*}
d_1(R,R') &= \inf\left\{|x-x'|:x\in J_1,~x'\in J'_1\right\}, \quad d_2(R,R')= \inf\left\{|y-y'|:y\in J_2,~y'\in J'_2\right\}.
\end{align*}Then we say that two rectangles $R$ and $R'$ are $C$-separated(with $C>0$) in the $i$th direction if $d_i(R,R') \geq C.$

\section{\bf Proof of Theorem \ref{thm3}}\
\subsection{\bf upper bound}\label{section1}

In this subsection, we will estimate the upper bound of $\hdim E_{\beta_1,\beta_2}(\hat{v},v)$. We start by recalling a classical result in the two-dimensional shrinking target problem in beta-dynamical systems.

\begin{lem}[\cite{HW18}]\label{measure}
Let $\beta_2\geq \beta_1>1$ and let $\psi_i$~(i = 1,2) be two positive functions defined on $\N.$ Let $W_{\beta_1,\beta_2}(\psi_1,\psi_2)$ be the set
\begin{equation*}
\left\{(x,y)\in [0,1]^2: \begin{array}{l}
T_{\beta_1}^n x<\psi_1(n) \\
T_{\beta_2}^n y<\psi_2(n)
\end{array}~ \text { for infinitely many }~ n\in\N \right\}.
\end{equation*}
Then
\begin{equation*}\label{msure_max_j} \mathcal{L}^{2}(W_{\beta_1,\beta_2}(\psi_1,\psi_2))=\left\{\begin{array}{ll}
   0, &\text{if}~\sum\limits_{n=1}^\infty \psi_1(n)\psi_2(n) < \infty, \\
   1, &\text{if}~\sum\limits_{n=1}^\infty \psi_1(n)\psi_2(n) =\infty,
   \end{array}
  \right.\end{equation*}
where $\L^2$ denotes the Lebesgue measure on $[0, 1]^2$.
  \end{lem}

We give the result for the case $v=0$.
\begin{lem}
One has $\mathcal{L}^2\{(x,y)\in[0,1]^2:v_{\beta_1,\beta_2}(x,y)=0 \}=1.$
\end{lem}
\begin{proof}It is evident that
\begin{align*}
& \left\{(x,y)\in [0,1]^2: v_{\beta_1,\beta_2}(x,y)>0\right\}\\
=&\bigcup_{m=1}^\infty \left\{(x,y)\in [0,1]^2:  v_{\beta_1,\beta_2}(x,y)>\frac{1}{m} \right\}\\
=&\bigcup_{m=1}^\infty \left\{(x,y)\in[0,1]^2: \left.\begin{array}{c}
T_{\beta_1}^nx < \beta_1^{- \frac{n}{m}} \\
T_{\beta_2}^ny< \beta_2^{-\frac{n}{m}}
\end{array} \right. ~ \text { for infinitely many }~ n\in\N  \right\}.
\end{align*}
The convergence of $\sum\limits_{n=1}^{\infty} (\beta_1\beta_2)^{-\frac{n}{m}} $ yields, via Lemma \ref{measure}, that
\begin{equation*}\L^2\left\{(x,y)\in [0,1]^2: v_{\beta_1,\beta_2}(x,y)>0\right\}=0 .
\end{equation*}
This completes the proof.
\end{proof}

We now turn to the case $0<v\leq \infty.$

\begin{lem}\label{lemex}
Let $(x,y)\in [0,1]^2$ and $v_{\beta_1,\beta_2}(x,y)>0$. If $(T_{\beta_1}^n x, T_{\beta_2}^n y)\neq (0,0)$ for all $n\in\N$, then there exist two sequences $\{n_k\}$ and $\{m_k\}$ depending on $(x,y)$ such that
\begin{equation}\label{jexpnt}
v_{\beta_1,\beta_2}(x,y) =\limsup_{k\rightarrow\infty} \frac{m_k-n_k}{n_k}=:\alpha(x,y)
\end{equation}
and
\begin{equation}\label{uexpnt}
\hat{v}_{\beta_1,\beta_2}(x,y)=\liminf_{k\rightarrow\infty} \frac{m_k-n_k}{n_{k+1}}=:\gamma(x,y) .
\end{equation}
\end{lem}

\begin{proof}
Define $l_{n,\beta}(x)$ as the maximal length of blocks of consecutive zeros just after the $n$th digit of $\varepsilon(x,\beta),$ i.e.,  \begin{equation*}\label{dfnln}l_{n,\beta}(x)=\sup\{k\geq0: \ep_{n+1}(x,\bt)=\cdots=\ep_{n+k}(x,\bt)=0 \}.\end{equation*}
For any $(x,y)\in [0,1]^2$, set
\begin{align*}
n_1= n_1(x,y) & =\inf\left\{n\geq 1: l_{n,\beta_1}(x)\cdot l_{n,\beta_2}(y)>0 \right\}, \\
m_1=m_1(x,y) & =n_1+\min \left\{l_{n,\beta_1}(x), l_{n,\beta_2}(y)\right\}.
\end{align*}

Assume that $n_k$ and $m_k$ have been defined for any $k\geq1$. Let
\begin{align*}
n_{k+1} &=n_{k+1}(x,y) =\inf\{n> m_k: l_{n,\beta_1}(x)\cdot l_{n,\beta_2}(y)>0 \}, \\
m_{k+1} &=m_{k+1}(x,y) =n_{k+1}+\min \{l_{n_{k+1},\beta_1}(x),l_{n_{k+1},\beta_2}(y)\}.
\end{align*}
Due to $v_{\beta_1,\beta_2}(x,y)>0$, $n_k$ is well defined.
Since $(T_{\beta_1}^n x, T_{\beta_2}^n y)\neq (0,0)$ for any $n\in\N$, $m_k$ is also well defined. Furthermore, $v_{\beta_1,\beta_2}(x,y)>0$ implies that $\limsup_{k\rightarrow \infty} (m_{k}-n_{k}) =\infty.$

Now we choose two subsequences $\{n_{i_k}\}$ and $\{m_{i_k}\}$ of $\{n_k\}$ and $\{m_k\}$ such that the sequence $\{m_{i_k}-n_{i_k}\}$ is non-decreasing. Let $i_1=1$. Suppose that $i_k$ has been defined. Let
\begin{equation*}
i_{k+1}=\min\{i> i_k : m_i-n_i\geq m_{i_k}-n_{i_k}\} .
\end{equation*}
For abbreviation, we continue to write
$\{n_k\}$ and $\{m_k\}$ for the subsequences $\{n_{i_k}\}$ and $\{m_{i_k}\}$ when there is no risk of confusion.

On the one hand, for any $\delta>0,$ there exists $k_0\geq1$ such that $m_k-n_k\geq (\gamma(x,y)-\delta)n_{k+1}$ for each $k\geq k_0$. Then, for any $N\geq n_{k_0}$, there exists $k\geq k_0$ such that $n_k\leq N <n_{k+1}$. Note that
\begin{equation*}
T^{n_k}_{\beta_1}x \leq \frac{1}{\beta_1^{l_{n_k,\beta_1}(x)}} \leq \frac{1}{\beta_1^{m_k-n_k}}\leq \frac{1}{\beta_1^{(\gamma(x,y)-\delta)n_{k+1}}}\leq\frac{1}{\beta_1^{(\gamma(x,y)-\delta)N}}.  \end{equation*}
Similarly, we obtain
\begin{equation*}
T^{n_k}_{\beta_2}y \leq \frac{1}{\beta_2^{l_{n_k,\beta_2}(y)}} \leq \frac{1}{\beta_2^{m_k-n_k}}\leq \frac{1}{\beta_2^{(\gamma(x,y)-\delta)N}}.
\end{equation*}
It follows that $\hat{v}_{\beta_1,\beta_2}(x,y)\geq\gamma (x,y)-\delta$ for any $\delta>0$.

On the other hand, for any $\delta>0$, there exists a subsequence $\{k_j\}$ such that
\begin{equation*}
m_{k_j}-n_{k_j}+1\leq (\gamma(x,y)+\delta)n_{{k_j}+1} .
\end{equation*}
From the definitions of the sequences $\{n_k\}_{k\geq1}$ and $\{m_k\}_{k\geq1}$, it follows that for any $1\leq n <n_{k+1}$,
\begin{equation*}
\min \left\{l_{n,\beta_1}(x),l_{n,\beta_2}(y)\right\}\leq m_k-n_k.
\end{equation*}
Thus, for any $T\geq1$, we can find $N:=n_{k_{j+1}}>T$ such that for all $1\leq n < N$, the following holds: if $l_{n,\beta_1}(x)=\min \{ l_{n,\beta_1}(x), l_{n,\beta_2}(y) \},$ then
\begin{equation*}
T_{\beta_1}^nx\geq \frac{1}{\beta_1^{ l_{n,\beta_1}(x) +1}}\geq \frac{1}{\beta_1^{ m_{k_j}-n_{k_j}+1}}\geq \frac{1}{\beta_1^{(\gamma(x,y)+\delta)n_{k_{j+1}}}} ,\end{equation*}
otherwise,
\begin{equation*}
 T_{\beta_2}^ny\geq \frac{1}{\beta_2^{ l_{n,\beta_2}(y)+1}}\geq \frac{1}{\beta_2^{ m_{k_j}-n_{k_j}+1}}\geq \frac{1}{\beta_2^{(\gamma(x,y)+\delta)n_{k_{j+1}}}}  .
\end{equation*}
Therefore, $\hat{v}_{\beta_1,\beta_2}(x,y)<\gamma(x,y)+\delta$ for any $\delta >0$. Hence, $\hat{v}_{\beta_1,\beta_2}(x,y)=\gamma(x,y)$.

The equality \eqref{jexpnt} follows from an analogous argument.
\end{proof}

Recall that
\begin{equation*}
E_{\beta_1,\beta_2}(\hat{v},v)=\{(x,y)\in[0,1]^2: \hat{v}_{\beta_1,\beta_2}(x,y)=\hat{v} ~\text{and}~ v_{\beta_1,\beta_2}(x,y)=v \}.
\end{equation*}
For any $(x,y)\in E_{\beta_1,\beta_2}(\hat{v},v)$ with $v>0$, if $(T_{\beta_1}^n x, T_{\beta_2}^n y)\neq (0,0)$ for all $n\in\N$, we associate $(x,y)$ with two sequences $\{n_k\}$, $\{m_k\}$ as in Lemma \ref{lemex}. It follows from the construction of the sequences $\{n_k\}$ and $\{m_k\}$ that
\begin{equation*}
m_k <n_{k+1} ~\text{for all}~ k\geq1 .
\end{equation*}
Thus
\begin{align}\label{tm1.6}
\hat{v}=\liminf\limits_{k\rightarrow \infty}\frac{m_k-n_k}{n_{k+1}}\leq  & \limsup\limits_{k\rightarrow \infty}\frac{m_k-n_k}{m_k} \\
= &1- \frac{1}{1+\limsup\limits_{k\rightarrow\infty}\frac{m_k-n_k}{n_k}}
= \frac{v }{1+v }, \nonumber
\end{align}
where $\frac{v}{1+v}=1~\text{if}~v=\infty$. Hence, if $0<\frac{v}{1+v}<\hat{v}\leq \infty$, then $E_{\beta_1,\beta_2}(\hat{v},v)$ is at most countable, and
\begin{equation*}
\hdim E_{\beta_1,\beta_2}(\hat{v},v)=0.
\end{equation*}

The task is now to construct a covering of $E_{\beta_1,\beta_2}(\hv,v) $ in the case $0\le \hat{v}\leq \frac{v}{1+v}<\infty$ and $0<v\leq \infty.$ Since $E_{\beta_1,\beta_2}(0,v)\subset \{(x,y)\in[0,1]^2: v_{\beta_1,\beta_2}(x,y)\geq v\}$, by Theorem \ref{wyf}, we get that
\begin{equation*} \hdim E_{\beta_1,\beta_2}(0,v) \le \left\{\begin{array}{ll}
\frac{2+v-v\log_{\beta_2}\beta_1}{1+v} &\text{if}~\beta_1^{1+v}<\beta_2,\\
\min\left\{\frac{1+\log_{\beta_2}\beta_1}{(1+v)\log_{\beta_2}\beta_1}, \frac{2+v-v\log_{\bt_2}\bt_1}{1+v}\right\} &\text{if}~\beta_1^{1+v}\geq\beta_2,
\end{array}
\right.\end{equation*}which is the desired upper bound.

It remains to deal with the case $0<\hat{v}\leq \frac{v}{1+v}<v<\infty.$ Since we focus on establishing the upper bound of $E_{\beta_1,\beta_2}(\hat{v},v),$ it suffices to construct a covering of
\begin{equation*}
\widetilde{E}_{\beta_1,\beta_2}(\hat{v},v):=E_{\beta_1,\beta_2}(\hat{v},v) \cap\left\{(x,y)\in[0,1]^2: ~(T_{\beta_1}^n x, T_{\beta_2}^n y)\neq (0,0)~\text{for all} ~n\in\N \right\} .\end{equation*}
Keep in mind that for all sequences $(\{n_k\},\{m_k\})$ associated with some $(x,y)\in \widetilde{E}_{\beta_1,\beta_2}(\hat{v},v),$ the following properties hold:

{\em{(i)}}~$x$ and $y$ have fixed blocks in the same positions, i.e.,
\begin{align}\label{Block}
&\varepsilon_{n_k+1}(x,\beta_1)=\cdots=\varepsilon_{m_k}(x,\beta_1)=0 ,~\text{and}\\
&\varepsilon_{n_k+1}(y,\beta_2)=\cdots=\varepsilon_{m_k}(y,\beta_2)=0 ,\nonumber
\end{align}
for all $k\geq1$.

{\em{(ii)}} For any $k\geq1$, we have
\begin{equation}\label{Ppstn}
n_k<m_k<n_{k+1} <m_{k+1}.
\end{equation}

{\em{(iii)}} The sequence $\{n_k\}$ increases exponentially. More specifically, there exists a constant $C>1$ independent of $(x,y)$ such that when $k$ is large enough,
\begin{equation*}
k\leq C \log n_k.
\end{equation*}
Indeed, it follows from $\hv>0$ that
$m_k-n_k\ge\frac{\hv}{2}n_{k+1}$ for all sufficiently large $k$, and thus for large enough $k$,
\begin{equation*}
\left(1-\frac{\hv}{2} \right)n_{k+1}\geq m_k-\frac{\hv}{2}n_{k+1}\geq n_k.
\end{equation*}

{\em{(iv)}} For any $\ep>0,$ we have that for large enough $k$,
\begin{equation}\label{keyeq1}
\sum_{i=1}^{k-1} (m_i-n_i)\geq \left(\frac{v\hv}{v-\hv }-\ep\right)n_k.
\end{equation}
In particular, we set $\frac{v\hv}{v-\hv}=\hv$ when $v=\infty.$

We conclude from \eqref{jexpnt} and \eqref{uexpnt} that for any $\delta>0$ there exists $k_0$ such that for any $k\geq k_0$,
\begin{equation*}
(\hat{v}-\delta)n_{k+1}\leq m_k-n_k\leq (v+\delta)n_k.
\end{equation*}
Thus, for any $\varepsilon>0$ there exist $\delta =\delta(\varepsilon)>0$ and $k_1>k_0$ such that for all $k>k_1$,
\begin{align*}
  \sum_{i=1}^{k-1}(m_i-n_i)&\geq \sum_{i=k_0}^{k-1}(\hat{v}-\delta)n_{i+1}\\
  & \ge \sum_{i=k_0}^{k-1}n_k(\hat{v}-\delta)\left(\frac{\hat{v}-\delta}{v+\delta}\right)^{k-i-1}\\
  & \ge \left(\frac{\hat{v} v}{v-\hat{v}}-\ep\right)n_k.
\end{align*}
Letting $\ep \to 0$, we have
\begin{equation*}
\liminf _{k \rightarrow \infty} \frac{\sum_{i=1}^{k-1} m_i-n_i}{n_k} \geq \frac{\hat{v} v}{v-\hat{v}}.
\end{equation*}

We now construct a covering of $\widetilde{E}_{\beta_1,\beta_2}(\hat{v},v)$ for $0<\hat{v}\leq \frac{v}{1+v}<\infty$ and $0<v\leq \infty$. We collect all sequences $\{n_k\}$ and $\{m_k\}$ associated with some $(x,y)\in \widetilde{E}_{\beta_1,\beta_2}(\hat{v},v)$ as in Lemma \ref{lemex} to form a set
\begin{equation*}
\Omega=\left\{(\{n_k\},\{m_k\}): \eqref{Ppstn}~\text{holds}, ~\limsup_{k\rightarrow\infty} \frac{m_k-n_k}{n_k}=v ~\text{and}~ \liminf_{k\rightarrow\infty} \frac{m_k-n_k}{n_{k+1}}=\hat{v} \right\} .
\end{equation*}
For $(\{n_k\},\{m_k\})\in \Omega,$ define
\begin{equation*}
F( \{n_k\},\{m_k\})=\{(x,y)\in [0,1]^2: \eqref{Block}~\text{holds} \},
\end{equation*}
\begin{equation*}
\Lambda_k,n_k(\ep)=\{(n_1,m_1;\cdots;n_{k-1},m_{k-1}): n_1<m_1<\cdots<m_{k-1}<n_k, ~\eqref{keyeq1}~\text{ holds}\},
\end{equation*}
and
\begin{align*}
  D_{n_1,m_1;\cdots;n_{k-1},m_{k-1}}=\{(\omega,\nu)\in \Sigma_{\bt_1}^{n_k}\times\ \Sigma_{\bt_2}^{n_k}:~& \omega_{n_i+1}=\cdots=\omega_{m_i}=0 , ~\text{and}~\\
 & \nu_{n_i+1}=\cdots=\nu_{m_i}=0 ,~\text{for all} ~1\leq i\leq k-1\}.
\end{align*}
For any $\omega\in \Sigma_{\bt}^n,$ let
\begin{equation*}
J_{n,\bt}^\ep(\omega) =\left\{x \in I_{n,\bt}(\omega): T_{\bt}^nx<\bt^{-n(v-\ep)}\right\} .
\end{equation*}
Thus, for any $\ep>0$,
\begin{align}\label{covering}
  &\widetilde{E}_{\bt_1,\bt_2}(\hv,v)\\  \nonumber
  \subset &\bigcup_{(\{n_k\},\{m_k\})\in \Omega} F( \{n_k\},\{m_k\})\\ \nonumber
  \subset& \bigcap_{K=1}^\infty\bigcup_{k=K}^\infty \bigcup_{n_k\geq  e^{\frac{k}{C}}}\bigcup_{(n_1,m_1,\cdots,n_{k-1},m_{k-1})\in \Lambda_{k,n_k}(\ep)} \bigcup_{(\omega,\nu)\in D_{n_1,m_1;\cdots;n_{k-1},m_{k-1}} } J_{n_k,\beta_1}^\ep(\omega)\times J_{n_k,\beta_2}^\ep(\nu).  \nonumber
\end{align}

It is straightforward to verify that
\begin{equation*}
\sharp \Lambda_{k,n_k}(\varepsilon)  \leq n_k^{2k} .
\end{equation*}
Applying \eqref{keyeq1}, $\bt_1\le\bt_2$, and Lemma \ref{number}, we deduce that if $(n_1,m_1,\cdots,m_{k-1})\in \Lambda_{k,n_k}(\ep),$ then
\begin{align*}
 \sharp D_{n_1,m_1;\cdots;n_{k-1},m_{k-1}} \leq & \left(\frac{\bt_1}{\bt_1-1}\right)^{k} \left(\frac{\bt_2}{\bt_2-1}\right)^{k} \bt_1^{ n_k- \sum_{i=1}^{k-1} (m_i-n_i)} \bt_2^{{ n_k- \sum_{i=1}^{k-1} (m_i-n_i)} } \\
 \leq  & (\bt_1-1)^{-2k}\bt_2^{k(1+\log_{\bt_2}\bt_1)}\bt_2^{n_k(1+\log_{\bt_2}\bt_1)\left(1-\frac{\hv v}{v-\hv}+\ep\right)}.
\end{align*}
Further, fixing $ \omega\in \Sigma_{\bt}^n,$ the equality \eqref{rep1} implies that for all $x,x'\in J_{n,\bt}^\ep(\omega),$
\begin{equation*}
|x-x'|\le \left|\frac{T_{\bt}^n x-T_{\bt}^n x'}{\bt^n} \right|\le \frac{2}{\bt^{n(1+v-\ep)}} ,
\end{equation*}
and thus
\begin{equation*}
|J_{n,\bt}^\ep(\omega)|\le \frac{2}{\bt^{n(1+v-\ep)}} .
\end{equation*}

According to \eqref{covering}, we observe that for any $K\ge1,$ the family
\begin{equation*}
\bigcup_{k=K}^\infty \bigcup_{n_k\geq  e^{\frac{k}{C}}}\bigcup_{\substack{\left(n_1, m_1; \cdots; n_{k-1},m_{k-1}\right) \in \Lambda_{k, n_k}(\varepsilon)}} \left\{ J^{\ep}_{n_k, \beta_1}(\omega) \times J^{\ep}_{n_k, \beta_2}(\nu): (\omega,\nu)\in D_{n_1,m_1;\cdots;n_{k-1},m_{k-1}}\right\}
\end{equation*}
forms a covering of $ \widetilde{E}_{\beta_1, \beta_2}(\hat{v}, v).$ We  proceed to cover  \begin{equation*}
\F_k:=\left\{ J_{n_k,\beta_1}^\ep(\omega)\times J_{n_k,\beta_2}^\ep(\nu): (\omega,\nu)\in D_{n_1,m_1;\cdots;n_{k-1},m_{k-1}}\right\}
\end{equation*}
and separate it into the following two cases.

{\em Case 1:}~If $\beta_1^{1+v}> \beta_2$, then choose $\ep>0$ such that $\bt_1^{1+v-\ep}\geq \bt_2.$ Thus, by \eqref{keyeq1}, in this case we get that
\begin{equation*}
\frac{1}{\bt_2^{( 1+v-\ep)n_k}}\le\frac{1}{\bt_1^{( 1+v-\ep)n_k}}\le \frac{1}{\bt_2^{n_k}}.
\end{equation*}
Hence, each rectangle $J_{n_k,\beta_1}^\ep(\omega)\times J_{n_k,\beta_2}^\ep(\nu)\in\F_k$ can be covered at most
\begin{equation*}
2+\frac{ 2\bt_1^{-n_k(1+v-\ep)} }{2\bt_2^{-n_k(1+v-\ep)}}\le 2\bt_2^{n_k(1+v-\ep)(1-\log_{\bt_2}\bt_1)}
\end{equation*}
many squares of side length $2\bt_2^{-n_k(1+v-\ep)}.$ Then, for any $s>(1+\log_{\beta_2}\beta_1)\left(\frac{v-\hat{v}-\hat{v} v}{(1+v-\ep)(v-\hat{v})}+2\varepsilon\right)+1-\log_{\beta_2}\beta_1$, we have
\begin{align}\label{e1}
&\mathcal{H}^s(\widetilde{E}_{\beta_1,\beta_2}(\hat{v},v))\\ \nonumber
\le &\liminf\limits_{K\rightarrow \infty} \sum_{k=K}^{\infty}\sum_{n_k=e^\frac{k}{C}}^\infty 2n_k^{2k} (\bt_1-1)^{-2k} \beta_2^{k(1+\log_{\beta_2}\beta_1)}\bt_2^{n_k(1+\log_{\bt_2}\bt_1)\left(1-\frac{\hv v}{v-\hv}+\ep\right)}\times \\ \nonumber
 &\qquad \qquad \qquad \qquad\bt_2^{n_k(1+v-\ep)(1-\log_{\bt_2}\bt_1)}\left( 2 \sqrt{2}\bt_2^{-(1+v-\ep)n_k} \right)^s \\ \nonumber
\ll& \liminf\limits_{K\rightarrow \infty}\sum_{k=K}^{\infty}\sum_{n_k=e^\frac{k}{C}}^\infty \beta_2^{n_k\left((1+\log_{\beta_2}\beta_1)\left(1-\frac{v\hv}{v-\hat{v}}+2\varepsilon\right)+(1+v-\ep)(1-\log_{\bt_2}\bt_1)-s(1+v-\ep)\right)}\\
\nonumber\le&\frac{1}{1-\bt_2^{t}}\sum_{k=1}^{\infty}(\beta_2^t)^{e^\frac{k}{C}}
<  \infty,
\end{align}where $t=\left((1+\log_{\beta_2}\beta_1)\left(1-\frac{v\hv}{v-\hat{v}}+2\varepsilon\right)+(1+v-\ep)(1-\log_{\bt_2}\bt_1)-s(1+v-\ep)\right)$. Here, $\mathcal{H}^s$ denotes the $s$-dimensional Hausdorff measure (see \cite{BP17}).
As a consequence, letting $\varepsilon\rightarrow 0$, we deduce that
\begin{align}\label{eq3009}
\hdim E_{\beta_1,\beta_2}(\hat{v},v)=&\hdim \widetilde{E}_{\beta_1,\beta_2}(\hat{v},v) \\
\leq & (1+\log_{\bt_2}\bt_1) \frac{v-\hv-\hv v}{(1+v)(v-\hv)} +1-\log_{\bt_2}\bt_1 .\nonumber
\end{align}

On the other hand, each rectangle $J_{n_k,\beta_1}^\ep(\omega)\times J_{n_k,\beta_2}^\ep(\nu)\in\F_k$ can be covered by a single square of side length $2\beta_1^{-n_k(1+v-\ep)}$. Thus, using the same method as in \eqref{e1}, for any $s>\frac{1+\log_{\bt_2}\bt_1}{\log_{\bt_2} \bt_1} \cdot \frac{v-\hv-\hv v+\ep(v-\hv)}{(1+v-\ep)(v-\hv)}$, we get
\begin{align*}
  \H^s(\widetilde{E}_{\bt_1,\bt_2}(\hv,v))&\leq \liminf\limits_{K\rightarrow \infty} \sum_{k=K}^{\infty}\sum_{n_k=e^\frac{k}{C}}^\infty n_k^{2k}(\bt_1-1)^{-2k} \bt_2^{k(1+\log_{\bt_2}\bt_1)} \\
  & \times \bt_2^{n_k(1+\log_{\bt_2}\bt_1)\left(1-\frac{\hv v}{v-\hv}+\ep\right)} \left(2\sqrt{2}\bt_2^{-n_k(1+v-\ep)\log_{\bt_2}\bt_1} \right)^s \\
  &<  \infty.
\end{align*}
Consequently,
\begin{align*}\label{eq3010}
\hdim E_{\beta_1,\beta_2}(\hat{v},v)=\hdim \widetilde{E}_{\beta_1,\beta_2}(\hat{v},v) \leq \frac{1+\log_{\beta_2}\beta_1}{\log_{\bt_2}\bt_1} \frac{v-\hat{v}-\hat{v} v}{(1+v)(v-\hat{v})}. \nonumber
\end{align*}
This combined with \eqref{eq3009} yields the desired upper bound for $\hdim E_{\beta_1,\beta_2}(\hat{v},v)$ in {\em Case 1}.

\smallskip
{\em Case 2:}~
If $\beta_1^{1+v}\leq \beta_2$, then for any $\ep>0$ we have $\frac{1}{\bt_2^{n_k}}<\frac{1}{\bt_1^{(1+v-\ep)n_k}}. $

Since $\bt_2^{-n_k(1+v-\ep)}\leq \bt_1^{-n_k(1+v-\ep)} ,$ as in {\em Case 1}, \eqref{eq3009} still holds. Moreover, each rectangle $J_{n_k,\beta_1}^\ep(\omega)\times J_{n_k,\beta_2}^\ep(\nu)\in \F_k$ can be covered by a single square of side length $2\bt_1^{-n_k(1+v-\ep)}.$

Combining this with $\bt_2^{-n_k} \leq \bt_1^{-n_k(1+v-\ep)},$ we conclude that each such square covers at least
\begin{equation*}
\frac{2\bt_1^{-n_k(1+v-\ep)}}{\bt_2^{-n_k} }-2\ge \frac{1}{2}\frac{2\bt_1^{-n_k(1+v-\ep)}}{\bt_2^{-n_k} } = \beta_2^{n_k(1-(1+v-\ep) \log_{\bt_2}\bt_1)}=:t_{n_k}(\omega)
\end{equation*}
many sets
\begin{equation*}
J_{n_k,\beta_1}^\ep(\omega)\times J_{n_k,\beta_2}^\ep(\nu_i)\subset J_{n_k,\beta_1}^\ep(\omega)\times I_{n_k,\beta_2}(\nu_i)~(i=1,\cdots,t_{n_k}(\omega)),
\end{equation*}
where $ I_{{n_k},\beta_2}(\nu_1),\cdots ,I_{{n_k},\beta_2}(\nu_{t_{n_k}(\omega)})$ are consecutive cylinders of order $n_k.$ Thus, the set
\begin{equation*}
\bigcup_{(\omega,\nu)\in D_{n_1,m_1;\cdots;n_{k-1},m_{k-1}} } J_{n_k,\beta_1}^\ep(\omega)\times J_{n_k,\beta_2}^\ep(\nu) \end{equation*}
can be covered by
\begin{align*}&\left(\frac{\bt_1}{\bt_1-1}\right)^{k} \left(\frac{\bt_2}{\bt_2-1}\right)^{k} \bt_1^{ n_k- \sum_{i=1}^{k-1} (m_i-n_i)} \bt_2^{{ n_k- \sum_{i=1}^{k-1} (m_i-n_i)} } \beta_2^{-n_k(1-(1+v-\ep) \log_{\bt_2}\bt_1)}\\
&\leq (\bt_1-1)^{-2k}  \bt_2^{k(1+\log_{\bt_2}\bt_1)}\bt_2^{n_k(1+\log_{\bt_2}\bt_1)\left(1-\frac{\hv v}{v-\hv}+\ep\right)}\beta_2^{-n_k(1-(1+v-\ep) \log_{\bt_2}\bt_1)}
\end{align*} many squares of side length $2\bt_1^{-{n_k}(1+v-\ep)}.$

Hence, for any $s>\frac{(3\ep+1)(v-\hv)-v\hv}{(1+v-\ep)(v-\hv) }+\frac{3\ep(v-\hv)-v\hv}{(1+v-\ep)(v-\hv)\log_{\bt_2}\bt_1}+1,$ we have
\begin{align*}
\H^s(\widetilde{E}_{\bt_1,\bt_2}(\hv,v))
&\leq \liminf\limits_{K\rightarrow \infty} \sum_{k=K}^{\infty}\sum_{n_k=e^\frac{k}{C}}^\infty n_k^{2k}\bt_2^{k(1+\log_{\bt_2}\bt_1)}\bt_2^{n_k(1+\log_{\bt_2}\bt_1)(1-\frac{\hv v}{v-\hv}+\ep)}\\
& \times  \beta_2^{-n_k(1-(1+v-\ep) \log_{\bt_2}\bt_1)} \left( 2 \sqrt{2}\bt_2^{-n_k(1+v-\ep)\log_{\bt_2}\bt_1} \right)^s\\
&< \infty.
\end{align*}
Therefore,
\begin{equation*}
\hdim E_{\beta_1,\beta_2}(\hat{v},v)=\hdim \widetilde{E}_{\beta_1,\beta_2}(\hat{v},v)\leq\frac{v-\hv-v\hv}{(1+v)(v-\hv)} -\frac{v\hv}{(1+v)(v-\hv)\log_{\bt_2}\bt_1}+1  .
\end{equation*}
Combining this result with \eqref{eq3009}, we obtain the required upper bound for $\hdim E_{\bt_1,\bt_2}(\hv,v)$ in {\em Case 2}.

\subsection{\bf lower bound}\

The method for determining the lower bound of $\hdim E_{\beta_1,\beta_2}(\hat{v},v)$ is classical. Firstly, we construct a Cantor
subset $F_\infty$ of the set $E_{\beta_1,\beta_2}(\hat{v},v)$; secondly, we define a suitable mass distribution $\mu$ supported on $F_\infty$; thirdly, we estimate the $\mu$-measure of a general ball; and finally, we apply the following mass distribution principle to complete the proof.
\begin{lem}[Mass distribution principle \cite{FA03}]\label{lem3003}
Let $\mu$ be a probability measure supported on a measurable set $\F$. Suppose there are positive constants $c$ and $r_0$ such that
\begin{equation*}
\mu(B(x,r))\le c r^s
\end{equation*}
for any ball $B(x,r)$ with radius $r\le r_0$ and center $x\in F$. Then $\mathcal{H}^s(F)\ge 1/c$ and thus $\hdim F\ge s$.
\end{lem}

\subsubsection{\bf Cantor subset construction}\

According to Theorem \ref{wyf} and the fact that $E_{\bt_1,\bt_2}(\hv,\infty)\subset
\bigcap\limits_{v>0} W_{\bt_1,\bt_2}(v) $, we obtain that
\begin{equation*}
\hdim E_{\bt_1,\bt_2}(\hv,\infty)=0.
\end{equation*}
Since $E_{\beta_1,\beta_2}(0,0)$ is of full Lebesgue measure and $\hdim E_{\beta_1,\beta_2}(\hat{v},v)=0$ for $\hat{v} >\frac{v}{1+v}$, we only need to consider the case $0\leq\hat{v}\leq \frac{v}{1+v}<v<\infty$.

For such $v$ and $\hat{v},$ choose two sequences $\{n_k\}_{k\geq1}$ and $\{m_k\}_{k\geq1}$ satisfying the following conditions:
\begin{enumerate}
  \item $n_k<m_k<n_{k+1}$ and $\{m_k-n_k\}_{k\geq1}$ is non-decreasing.
  \item One has \begin{equation}\label{as1} \lim\limits_{k\rightarrow \infty} \frac{m_k-n_k}{n_k}=v~\text{and}~\lim\limits_{k\rightarrow \infty} \frac{m_k-n_k}{n_{k+1}}=\hat{v}.\end{equation}

\end{enumerate}
Indeed, when $\hat{v}>0$, we can take
\begin{equation*}
n_1=1,~ n_{k+1}=\left\lfloor \frac{v}{\hat{v}}n_k\right \rfloor+2,~ m_k=\lfloor (1+v)n_k \rfloor+1.
\end{equation*}
When $\hat{v}=0$, we can take
\begin{equation*}
n_k=2^{2^k},~ m_k=\lfloor (1+v)n_k \rfloor+1.
\end{equation*}

Recall that $\beta_2\geq \beta_1>1$ and for each $i=1,2$, $\beta_{i,N}$ is the unique real number satisfying the equation
\begin{equation*}
1=\frac{\varepsilon^*_1(1,\beta_i)}{z}+\cdots+\frac{\varepsilon^*_N(1,\beta_i)}{z^N},
\end{equation*}
where $\varepsilon^*_N(1,\beta_i) > 0.$

We will construct a tree-like Cantor set $F_\infty$ by addressing the $\beta$-expansion sequences of each point in $F_\infty$ step by step. To this end, we construct a sequence of families $\{\F_n\}_{n \geq 1}$ consisting of Cartesian products of cylinders, and the Cantor set is defined as
\begin{equation*}
F_\infty = \bigcap_{n=1}^\infty \bigcup_{Q_n \in \F_n}  Q_n.
\end{equation*}

\textit{The first level of the Cantor set.}

\textbullet Step $1_1$. Construct rectangles.

Define $l_1=n_1+2N+1$ and $h_1=l_1+m_1-n_1+2N+1.$ Let
\begin{equation*}
B_1(\beta_1)=\{\omega_1=\xi_1 0^N10^N\varepsilon_1 0^N10^N  \in\Sigma_{\beta_{1,N}}^{h_1}: \varepsilon_1=\underbrace{0\cdots 0}_{m_1-n_1 },\xi_1 \in\Sigma_{\beta_{1,N}}^{n_1}\},
\end{equation*}
and
\begin{equation*}
B_1(\beta_2)=\{\nu_1=\zeta_1 0^N10^N\varepsilon_1 0^N10^N  \in\Sigma_{\beta_{2,N}}^{h_1}: \varepsilon_1=\underbrace{0\cdots 0}_{m_1-n_1 },\zeta_1 \in\Sigma_{\beta_{2,N}}^{n_1}\}.
\end{equation*}
Here, the block $0^N$ is inserted to guarantee that the concatenation $\omega0^N\ep$ is still $\beta_i$-admissible for $i=1,2$ for any $\omega\in \Sigma_{\beta_{i,N}}^{n}$ and $\varepsilon\in \Sigma_{\beta_{i,N}}^{m}$~(see Lemma \ref{lem1}). The digit $1$ is inserted to control the run length of consecutive zeros.

For each $(\omega_1,\nu_1) \in B_1(\beta_1) \times B_1(\beta_2)$, the rectangle
\begin{equation*}
R_1=R_1(\omega_1,\nu_1):=I_{h_1,\beta_1} (\omega_1) \times I_{h_1,\beta_2}(\nu_1)
\end{equation*}
is called a basic rectangle of order $1.$ Then, define $\mR_1(Q_0) $ as the family of basic rectangles generated by $Q_0=[0,1]^2,$ that is,
\begin{equation*}
\mR_1(Q_0)=\left\{R_1(\omega_1,\nu_1):  (\omega_1,\nu_1) \in B_1(\beta_1) \times B_1(\beta_2)\right\} .
\end{equation*}

\medskip

\textbullet Step $2_1$. Dividing.

Let \begin{equation}\label{choose1}\tilde{h}_1= \lfloor h_1 \log_{\bt_1} \bt_2\rfloor+2N+1\end{equation} and choose $\tilde{t}_1\ge0$ and $0\leq \tilde{r}_1 < m_1-n_1$ such that
\begin{equation*}
\tilde{h}_1=h_1+\tilde{t}_1(m_1-n_1+2N+1)+\tilde{r}_1 +2N+1.
\end{equation*}
Then, for any $\omega_1\in B_{1}(\bt_1),$ set
\begin{align*}
D_1(\beta_1,\omega_1)=\{{\tilde{\omega}_1}=&\omega_1 (\xi_1' 0^N10^N) \cdots (\xi_{\tilde{t}_1}' 0^N10^N) (\tau_{1}' 0^N10^N)\in\Sigma_{\beta_{1,N}}^{\tilde{h}_1}:\\ &\omega_1 \in B_1(\beta_1),~\tau_{1}'\in \Sigma_{\beta_{1,N}}^{\tilde{r}_1 },~\xi_j' \in\Sigma_{\beta_{1,N}}^{m_1-n_1} ~\text{for}~ 1\leq j \leq \tilde{t}_1 \}.
\end{align*}

For each $(\tilde{\omega}_1,\nu_1) \in D_1(\beta_1,\omega_1) \times B_1(\beta_2)$, it follows from \eqref{choose1} that
\begin{equation*} |I_{\tilde{h}_1,\beta_1} (\tilde{\omega}_1)|=\frac{1}{\bt_1^{\tilde{h}_1}} \asymp\frac{1}{\bt_2^{h_1}} =|I_{h_1,\beta_2}(\nu_1)|.\end{equation*}
The implied constant in the $\asymp$ depends only on $N,$ $\beta_1$ and $\beta_2.$ Then, we write
\begin{equation*}
Q_1=Q_1(\tilde{\omega}_1,\nu_1):=I_{\tilde{h}_1,\beta_1} (\tilde{\omega}_1) \times I_{h_1,\beta_2}(\nu_1)
\end{equation*}
and call the square $Q_1$ a basic square of order $1.$
Denote by $\mQ_1(R_1)$ the family of the basic squares $Q_1\subset R_1$, i.e.,
\begin{equation*}
Q_1(R_1)=\left\{Q_1(\tilde{\omega}_1,\nu_1):  (\tilde{\omega}_1,\nu_1) \in D_1(\beta_1,\omega_1) \times B_1(\beta_2)\right\} .
\end{equation*}

With these notations, the first sub-level $\F_1$ is defined as
\begin{align*}
  \F_1=&\{Q_1\in\mQ_1(R_1):R_1\in \mR_1([0,1]^2)\}  \\
  =& \{Q_1(\tilde{\omega}_1,\nu_1) : (\omega_1,\nu_1)\in B_1(\beta_1)\times B_1(\beta_2)~\text{and}~\tilde{\omega}_1\in D_1(\beta_1,\omega_1)\}.
\end{align*}

\textit{From Level $k-1$ to Level $k.$}

Suppose that $\F_{k-1}$ has been defined. For each basic square
$Q_{k-1}\in\F_{k-1},$ we will define a sub-family $\F_k(Q_{k-1}),$ and then let
\begin{equation*}
\F_k = \left\{Q_k \in \F_k(Q_{k-1}) : Q_{k-1}\in \F_{k-1}\right\}.
\end{equation*}
Let $Q_{k-1}$ be an arbitrary square in $\mathcal{F}_{k-1},$ written as
\begin{equation*}
Q_{k-1}:=I_{\tilde{h}_{k-1},\bt_1}(\tilde{\omega}_{k-1})\times  I_{h_{k-1},\bt_2}(\nu_{k-1}).
\end{equation*}

\textbullet Step $1_{k}$. Construct rectangles.

Take integers $t_{k-1}\ge 0$ and $0\leq r_{k-1} < m_{k-1}-n_{k-1}$ such that
\begin{equation*}
n_k=m_{k-1} +t_{k-1}(m_{k-1}-n_{k-1}) +r_{k-1}.
\end{equation*}
Then let
\begin{align*}
l_{k}=&h_{k-1}+t_{k-1}(m_{k-1}-n_{k-1}+2N+1)+r_{k-1}+2N+1\\
=&n_{k} +2N+1+\sum_{i=1}^{k-1} (t_i+2)(2N+1)
\end{align*}and
\begin{align*}
h_{k}=l_{k}+m_{k}-n_{k}+2N+1=m_{k} +2(2N+1)+\sum_{i=1}^{k-1} (t_i+2)(2N+1).
\end{align*}
Note that $\log _{\beta_2}{\beta_1}>\frac{\hat{v}}{v}(1+v)$ and the equalities in \eqref{as1} hold. By passing to a subsequence if necessary, we may assume that $\{n_k\}_{k\geq1}$ and $\{m_k\}_{k\geq1}$ are sufficiently sparse that
\begin{equation*}
\frac{1}{\bt_1^{l_k}} \le \frac{1}{\bt_1^{\tilde{h}_{k-1}+2N+1}}~\text{for all }~k\ge 2.
\end{equation*}
Thus, there exist $\bar{t}_{k-1}\ge 0$ and $0\leq \bar{r}_{k-1} < m_{k-1}-n_{k-1}$ such that
\begin{equation}\label{as2} l_k=\tilde{h}_{k-1}+\bar{t}_{k-1}(m_{k-1}-n_{k-1}+2N+1) +\bar{r}_{k-1}+2N+1.\end{equation}

Then, set
\begin{align*}
&B_{k}(\beta_1,\tilde{\omega}_{k-1})\\
=&\Big\{\omega_k=\tilde{\omega}_{k-1}(\xi_1 0^N10^N) \cdots (\xi_{\bar{t}_{k-1}}0^N10^N) (\tau_{k-1}0^N10^N) (\varepsilon_{k}0^N10^N) \in \Sigma_{\beta_{1,N}}^{h_{k}}: \\
&\quad \tilde{\omega}_{k-1}\in D_{k-1}(\beta_1,\omega_{k-1} ),~\varepsilon_{k}=\underbrace{0\cdots 0}_{m_{k}-n_{k} },~
 \tau_{k-1}\in \Sigma_{\beta_{1,N}}^{\bar{r}_{k-1}},~\xi_j\in \Sigma_{\beta_{1,N}}^{m_{k-1}-n_{k-1}} ~\text{for}~ 1\leq j \leq \bar{t}_{k-1}\Big\},
\end{align*}
and
\begin{align*}
&B_k(\beta_2,\nu_{k-1})\\
=&\big\{\nu_k=\nu_{k-1}(\zeta_1 0^N10^N) \cdots (\zeta_{t_{k-1}}0^N10^N) (\gamma_{k-1}0^N10^N) (\varepsilon_{k}0^N10^N) \in \Sigma_{\beta_{2,N}}^{h_k}:\\&\quad \nu_{k-1} \in B_{k-1}(\beta_2,\nu_{k-2}),~\varepsilon_k=\underbrace{0\cdots 0}_{m_k-n_k},~\gamma_{k-1}\in \Sigma_{\beta_{2,N}}^{r_{k-1}},~\zeta_j \in\Sigma_{\beta_{2,N}}^{m_{k-1}-n_{k-1}} ~\text{for}~ 1\leq j \leq t_{k-1} \big\}.
\end{align*}

For each $(\omega_k,\nu_k) \in B_{k}(\beta_1,\tilde{\omega}_{k-1}) \times B_k(\beta_2,\nu_{k-1})$,
the rectangle
\begin{equation*}
R_k=R_k(\omega_k,\nu_k):=I_{h_k,\beta_1} (\omega_k) \times I_{h_k,\beta_2}(\nu_k)
\end{equation*}
is called a basic rectangle of order $k.$ Denote by $\mR_k(Q_{k-1}) $ the family of basic rectangles generated by $Q_{k-1},$ i.e.,
\begin{equation*}
\mR_k(Q_{k-1})=\left\{R_k(\omega_k,\nu_k):  (\omega_k,\nu_k) \in B_k(\beta_1,\tilde{\omega}_{k-1}) \times B_k(\beta_2,\nu_{k-1})\right\} .
\end{equation*}

\textbullet Step $2_{k}$. Dividing.

Define
\begin{equation}\label{as3}
\tilde{h}_k = \left\lfloor h_k \log_{\bt_1} \bt_2 \right\rfloor+2N+1,
\end{equation}
and select $\tilde{t}_k \ge 0$ and $0 \leq \tilde{r}_k < m_k - n_k$ such that
\begin{equation*}
\tilde{h}_k = h_k + \tilde{t}_k(m_k - n_k + 2N + 1) + \tilde{r}_k + 2N + 1.
\end{equation*}
Fix $\omega_k \in B_{k}(\beta_1,\tilde{\omega}_{k-1})$, and set
\begin{align*}
D_k(\beta_1,\omega_k)=\{{\tilde{\omega}_k}=&\omega_k  (\xi_k' 0^N10^N) \cdots (\xi_{\tilde{t}_k}' 0^N10^N) (\tau_{k}' 0^N10^N)\in\Sigma_{\beta_{1,N}}^{\tilde{h}_k}:\\
&\omega_k \in B_{k}(\beta_1,\tilde{\omega}_{k-1}),~\tau_{k}'\in \Sigma_{\beta_{1,N}}^{\tilde{r}_k},~\xi_j' \in\Sigma_{\beta_{1,N}}^{m_k-n_k} ~\text{for}~ 1\leq j \leq \tilde{t}_k \}.
\end{align*}

For each $(\tilde{\omega}_k,\nu_k) \in D_k(\beta_1,\omega_k) \times B_k(\beta_2,\nu_{k-1})$, we define
\begin{equation*}
Q_k=Q_k(\tilde{\omega}_k,\nu_k):=I_{\tilde{h}_k,\beta_1} (\tilde{\omega}_k) \times I_{h_k,\beta_2}(\nu_k)
\end{equation*}
and we call the square $Q_k(\tilde{\omega}_k,\nu_k)$ a basic square. Note that
$ Q_k(\tilde{\omega}_k,\nu_k)\subset R_k(\omega_k,\nu_k).$  Then, for each $R_{k}=R_k(\omega_k,\nu_k)\in\mR_k(Q_{k-1}),$ denote the family of  basic squares  $Q_k \subset R_{k}$ by $\mQ_k(R_{k}) $, i.e.,
\begin{equation*}
\mQ_k(R_{k})=\left\{Q_k(\tilde{\omega}_k,\nu_k): (\tilde{\omega}_k,\nu_k) \in D_k(\beta_1,\omega_{k})\times B_k(\beta_2,\nu_{k-1}) \right\} .
\end{equation*}

Then the $k$th level, $\F_k$, is defined as
\begin{equation*}
\F_{k}=\left\{\F_{k}(Q_{k-1}): Q_{k-1}\in \F_{k-1}\right\},
\end{equation*}
where, for a fixed $Q_{k-1}=Q_{k-1}(\tilde{w}_{k-1},v_{k-1})\in \F_{k-1},$
\begin{align*}\F_{k}(Q_{k-1})=&\left\{Q_k: Q_k\in \mQ_k(R_k)~\text{and}~R_k \in \mR_k(Q_{k-1})\right\} \\
=&\left\{Q_k(\tilde{\omega}_k,\nu_k): (\omega_k, \nu_k) \in B_{k}\left(\beta_1,\tilde{w}_{k-1}\right) \times B_{k}\left(\beta_2,v_{k-1}\right)~\text{and}~ \tilde{\omega}_k \in D_{k}\left(\beta_1,w_{k}\right)\right\}.\end{align*}

Finally, the desired Cantor set is defined as
\begin{equation*}
F_\infty:=\bigcap_{k\geq1}\bigcup_{Q_k\in \F_k}Q_{k}.
\end{equation*}

\begin{rem}Under the assumption $\log_{\beta_2}{\beta_1} > \frac{\hat{v}}{v}(1 + v),$ we have \eqref{as2} holds. Consequently, the Cantor subset $F_\infty$ is well defined. However, without the condition $\log_{\beta_2}{\beta_1} > \frac{\hat{v}}{v}(1 + v),$ the lower bound estimate of $\hdim E_{\beta_1,\beta_2}(\hat{v},v)$ is not known, and the methods developed in the present paper are likely to be ineffective.
\end{rem}

Subsequently, we summarise the properties of the Cantor set $F_\infty$ for later use.

\begin{lem}\label{lengap} Let $k\ge 1$ and $Q_{k-1}\in \F_{k-1}.$
  \begin{enumerate}

    \item\label{len0} For any $R_k\in \mR_k(Q_{k-1})$ and $Q_k \in \mQ_k(R_k),$ they are of the form
        \begin{equation*} R_k=I_{h_k,\beta_1} (\omega_k) \times I_{h_k,\beta_2}(\nu_k)~\text{and}~ Q_k=I_{\tilde{h}_k,\beta_1} (\tilde{\omega}_k) \times I_{h_k,\beta_2}(\nu_k),\end{equation*}
    with lengths satisfying
\begin{equation*}\frac{1}{\beta_1^{\tilde{h}_k}}=|I_{\tilde{h}_k,\beta_1}(\tilde{\omega}_k)|\asymp|I_{h_k,\beta_2} (\nu_k) |=\frac{1}{\bt_2^{h_k}}< \frac{1}{\bt_1^{h_k}}
 =|I_{h_k,\beta_1} (\omega_k)|  .\end{equation*}

  \item\label{gap0} The rectangles in $\mR_k(Q_{k-1})$ are $\frac{1}{\bt_i^{l_k}}$-separated in the $i$th direction; the basic squares in $\mQ_k(R_k)$  are $\frac{1}{\beta_1^{\tilde{h}_k}}$-separated in the first direction.

    \item  One has \begin{equation*}\sharp \mR_k(Q_{k-1})=(\sharp \Sigma_{\beta_{1,N}}^{m_{k-1}-n_{k-1}})^{\bar{t}_{k-1}} \cdot \Sigma_{\beta_{1,N}}^{\bar{r}_{k-1}} \cdot (\sharp \Sigma_{\beta_{2,N}}^{m_{k-1}-n_{k-1}})^{t_{k-1}} \cdot \Sigma_{\beta_{2,N}}^{r_{k-1}} \end{equation*} and for any $R_k \in \mR_k(Q_{k-1}),$
\begin{equation*}\sharp \mQ_k(R_{k})=(\Sigma_{\beta_{1,N}}^{m_{k-1}-n_{k-1}})^{\tilde{t}_{k}} \cdot \Sigma_{\beta_{1,N}}^{\tilde{r}_{k}}.
\end{equation*}
  \end{enumerate}
\end{lem}
\begin{proof}
These results follow directly from \eqref{as3}, Lemma \ref{lem1} and the construction of Cantor set $F_\infty.$
\end{proof}

\begin{pro}
One has
$
F_{\infty} \subset E_{\beta_1, \beta_2}(\hat{v}, v).
$
\end{pro}
\begin{proof}
By the construction of $ F_{\infty} $ and the same method as in the proof of Lemma \ref{lemex}, it follows that for any $(x,y)\in F_\infty,$
\begin{equation*}
\hat{v}_{\beta_1,\beta_2}(x,y)=\lim\limits_{k\rightarrow \infty}\frac{h_k-l_k-1}{l_{k+1}-N}= \lim\limits_{k\rightarrow \infty} \frac{m_k-n_k+2N}{n_{k+1}+\sum_{i=1}^{k} (t_i+2)(2N+1)+N+1},
\end{equation*}
and
\begin{equation*}
v_{\beta_1,\beta_2}(x,y)=\lim\limits_{k\rightarrow \infty}\frac{h_k-l_k-1}{l_{k}-N}=\lim\limits_{k\rightarrow \infty} \frac{m_k-n_k+2N}{n_{k}+\sum_{i=1}^{k-1} (t_i+2)(2N+1)+N+1} .
\end{equation*}
Thus, by applying the Stolz-Ces\`{a}ro Theorem,
\begin{equation}\label{eq31101}
\lim _{k \rightarrow \infty} \frac{\sum_{i=1}^k t_i}{n_{k+1}}=\lim _{k \rightarrow \infty} \frac{t_k}{n_{k+1}-n_k}=0.
\end{equation}
Hence,
\begin{equation}\label{e2} \hat{v}_{\beta_1,\beta_2}(x,y)=\lim\limits_{k\rightarrow \infty}\frac{h_k-l_k-1}{l_{k+1}-N}=\lim\limits_{k\rightarrow \infty} \frac{m_k-n_k}{n_{k+1}}=\hat{v},
\end{equation}
and
\begin{equation}\label{e3}
v_{\beta_1,\beta_2}(x,y)=\lim\limits_{k\rightarrow \infty}\frac{h_k-l_k-1}{l_{k}-N}=\lim\limits_{k\rightarrow \infty} \frac{m_k-n_k}{n_k}=v .
\end{equation}
Therefore,
\begin{equation*}
F_\infty\subset E_{\beta_1,\beta_2}(\hat{v},v).
\end{equation*}
\end{proof}
\subsubsection{\bf Measure distribution}\

The mass distribution to be defined comes from the following consideration:

\textbullet Since the measure is supported on $F_{\infty}$, the total measure of the rectangles $R_k$  generated by the same basic square $Q_{k-1}$ is equal to the measure of this square. Note that such rectangles $R_k$ are of the same size. Therefore, it is natural to distribute the mass of $Q_{k-1}$ equally to the rectangles $R_k\in \mR_k(Q_{k-1})$.

\textbullet For the same reason, the measure of $R_k$ is equally distributed to $\mQ_k(R_k)$.

In the following we distribute a measure $\mu$ on $F_{\infty}$.

\textit{Measure of basic squares in $\mathcal{F}_1$.}
The measure of $R_1\in \mR_1(Q_0)$ is defined as
\begin{equation*}
\mu(R_{1} ):= \frac{1}{\sharp \mR_1(Q_0)}=\frac{1}{ \sharp\Sigma_{\beta_{1,N}}^{n_1}\cdot \sharp\Sigma_{\beta_{2,N}}^{n_1}} ,
\end{equation*}
and for any $Q_1 \in \mQ_1(R_1)$, we define
\begin{equation*}
\mu\left(Q_1\right):=\frac{\mu\left(R_1\right)}{\sharp \mQ_1(R_1)} =\frac{1}{ \sharp\Sigma_{\beta_{1,N}}^{n_1}\cdot \sharp\Sigma_{\beta_{2,N}}^{n_1}\cdot (\sharp \Sigma_{\beta_{1,N}}^{m_{1}-n_{1}})^{\tilde{t}_{1}} \cdot \sharp \Sigma_{\beta_{1,N}}^{\tilde{r}_{1}} } .
\end{equation*}

\textit{Measure of basic squares in $\mathcal{F}_k.$} Assume that the measure of basic squares $Q_{k-1}\in \F_{k-1}$ has already been defined. For any basic rectangle $R_{k}\in \F_k,$ there exists a unique basic square $Q_{k-1}\in\F_{k-1}$ such that $R_{k}\in\mR_k(Q_{k-1}).$ Then, we define
\begin{align*}
\mu\left(R_{k}\right)
:=&\frac{\mu\left(Q_{k-1}\right)}{\sharp \mR_k(Q_{k-1})}\\
=&\frac{1}{ \sharp\Sigma_{\beta_{1,N}}^{n_1}\cdot \sharp\Sigma_{\beta_{2,N}}^{n_1} \cdot \prod_{l=1}^{k-1} [ ( \sharp \Sigma_{\beta_{1,N}}^{m_l-n_l} )^{\tilde{t}_{l}+\bar{t}_{l}}\cdot \sharp\Sigma_{\beta_{1,N}}^{\tilde{r}_l}\cdot\sharp\Sigma_{\beta_{1,N}}^{\bar{r}_l}\cdot(\sharp \Sigma_{\beta_{2,N}}^{m_l-n_l} )^{t_{l}}\cdot \sharp\Sigma_{\beta_{2,N}}^{r_l}]}.
\end{align*}
For any basic square $Q_k\in \mQ_k(R_k),$ we set
\begin{equation*}
\mu(Q_{k} ):=\frac{\mu(R_k)}{\sharp \mQ_k(R_k)}=\frac{\mu(R_k)}{(\sharp \Sigma_{\beta_{1,N}}^{m_{k-1}-n_{k-1}})^{\tilde{t}_{k}} \cdot \sharp\Sigma_{\beta_{1,N}}^{\tilde{r}_{k}} } .
\end{equation*}

For any $n\ge1$ and any $n$th order rectangle $I_{n,\beta_1}(\omega_n) \times I_{n,\beta_2}(\nu_n)$ with $ I_{n,\beta_1}(\omega_n) \times I_{n,\beta_2}(\nu_n) \cap \F_\infty \neq \emptyset$, let $k \geq 0$ denote the integer such that $h_k < n\leq h_{k+1}$ (by setting $h_0=0$). To reach a measure, we set
\begin{equation*}
\mu(I_{n,\beta_1}(\omega_n) \times I_{n,\beta_2}(\nu_n) ) =\sum_{R_{k+1} \subset I_{n,\beta_1}(\omega_n) \times I_{n,\beta_2}(\nu_n) } \mu(R_{k+1}),
\end{equation*}
where the summation is taken over all basic rectangles $R_{k+1}\in \F_{k+1}$ contained in $ I_{n,\beta_1}(\omega_n) \times I_{n,\beta_2}(\nu_n)$.

Then by Kolmogorov's consistency theorem, $\mu$ can be uniquely extended to a probability measure supported on $F_{\infty}$.

\subsubsection{\bf Measure of basic squares in $\F_k$ }\

To apply the mass distribution principle (Lemma \ref{lem3003}), we first analyze the relationship between $\mu(Q_k)$ and $|Q_k|$.

By \eqref{as3} and Lemma \ref{lengap} \eqref{len0}, we conclude that
\begin{equation*}
\frac{1}{\bt_2^{h_k}} \cdot \frac{\sqrt{2}}{\bt_1^{2N+2}}\le |Q_k|=\left|I_{\tilde{h}_k,\beta_1} (\tilde{\omega}_k) \times I_{h_k,\beta_2}(\nu_k) \right|\le \frac{\sqrt{2}}{\bt_2^{h_k}}.
\end{equation*}
From Lemma \ref{number}, it follows that
\begin{align}\label{32303}
&\liminf\limits_{k\rightarrow\infty} \frac{\log\mu(Q_k)}{\log |Q_k|} \\
= & \liminf\limits_{k\rightarrow \infty} \frac{1}{h_k}\cdot\left[\log_{\beta_2} \sharp \Sigma _{\beta_{1,N}}^{n_1}+\log_{\beta_2} \sharp \Sigma _{\beta_{2,N}}^{n_1}+ \tilde{t}_k \log_{\beta_2} \sharp\Sigma_{\beta_{1,N}}^{m_{k-1}-n_{k-1}}+\log_{\beta_2} \sharp\Sigma_{\beta_{1,N}}^{\tilde{r}_k} \right. \nonumber\\
&+\left.\sum_{l=1}^{k-1} ((\tilde{t}_l+\bar{t}_l) \log_{\beta_2} \sharp \Sigma_{\beta_{1,N}}^{m_l-n_l}
+ \log_{\beta_2} \sharp \Sigma_{\beta_{1,N}}^{\tilde{r}_l}+\log_{\beta_2} \sharp \Sigma_{\beta_{1,N}}^{\bar{r}_l} +t_l \log_{\beta_2} \sharp \Sigma_{\beta_{2,N}}^{m_l-n_l}+\log_{\beta_2} \sharp \Sigma_{\beta_{2,N}}^{r_l})  \right] \nonumber\\
=& \liminf\limits_{k\rightarrow \infty} \frac{(l_k-\sum_{i=1}^{k-1}(h_i-l_i)- 2N-1)(\log_{\beta_2} {\beta_{1,N}}+\log_{\beta_2} {\beta_{2,N}} )}{h_k} \nonumber\\
&- \liminf\limits_{k\rightarrow \infty}\frac{\sum\limits_{i=1}^{k-1}[(\tilde{t}_i+\bar{t}_i+2)(2N+1)\log_{\bt_2}\bt_{1,N}+(t_i+1)(2N+1)\log_{\bt_2}\bt_{2,N}]}{h_k}\nonumber\\
&+ \liminf\limits_{k\rightarrow \infty}\frac{[\tilde{h}_k -h_k-(\tilde{t}_k+1)(2N+1)] \log_{\bt_2} \bt_{1,N}}{h_k}.\nonumber
\end{align}

Using the Stolz-Ces\`{a}ro Theorem, we can assert that
\begin{align}\label{32302}
\lim\limits_{k\rightarrow\infty} \frac{l_k-\sum_{i=1}^{k-1}(h_i-l_i)}{h_k} =&\lim\limits_{k\rightarrow \infty} \frac{ l_1+\sum_{i=1}^{k-1} (l_{i+1}-h_i)}{h_k} \nonumber \\
=& \lim\limits_{k\rightarrow \infty} \frac{l_k-h_{k-1}}{h_k-h_{k-1}} \nonumber\\
=& \frac{v-\hat{v}-\hat{v} v}{(1+v)(v-\hat{v})},
\end{align}
where the last equality follows from \eqref{e2} and \eqref{e3}. Note that $\tilde{t}_k+\bar{t}_k\le t_k+2$ for all $k\ge1$. Combining \eqref{as3}, \eqref{eq31101}, \eqref{32303} and \eqref{32302}, we have
\begin{equation*}
\liminf\limits_{k\rightarrow\infty} \frac{\log\mu(Q_k)}{\log |Q_k|}=\frac{(\log_{\beta_2} {\beta_{1,N}}+\log_{\beta_2} {\beta_{2,N}} )(v-\hat{v}-\hat{v} v)}{(1+v)(v-\hat{v})}+1-\log_{\beta_2}\beta_1.
\end{equation*}

Fix $\delta>0$, and let $s=\frac{(\log_{\beta_2} {\beta_{1}}+1 )(v-\hat{v}-\hat{v} v)}{(1+v)(v-\hat{v})}+ 1-\log_{\beta_2}\beta_1 -\delta$. Hence for any $k \gg 1$, we have
\begin{equation}\label{32301}
\mu\left(Q_k\right) \leq\left|Q_k\right|^{s}\ll{\beta_2}^{-sh_k}.
\end{equation}

\subsubsection{\bf Measure of a general ball}\

Let $\ep>0.$ Since \eqref{e3} holds, there exists $k_0\in\N$ such that for any $k\ge k_0$,
\begin{equation*}  \left\{\begin{array}{ll}
\frac{1}{\bt_1^{h_k}}<\frac{1}{\bt_2^{l_k}}  &\text{if}~\beta_1^{1+v}>\beta_2;\\
\frac{1}{\bt_1^{h_k}}> \frac{1}{\bt_2^{l_k}}  &\text{if}~\beta_1^{1+v}< \beta_2;\\
\frac{1}{\bt_2^{l_k(1+\ep)}}\le\frac{1}{\bt_1^{h_k}}\le \frac{1}{\bt_2^{l_k(1-\ep)}}  & \text{if}~\beta_1^{1+v}= \beta_2.
\end{array}
\right.\end{equation*}Set
$ r_0=\frac{1}{\bt_1^{\tilde{h}_{k_0}}}.$ For any $\mathbf{z}=\left(x, y\right) \in F_{\infty}$ and  $r<r_0,$ we estimate the measure of the ball $B(\mathbf{z}, r).$

Since $\mu$ is a probability measure supported on the set $F_{\infty}$, $\mu(B(\mathbf{z}, r))=0$ whenever $B(\mathbf{z}, r) \cap F_{\infty} = \emptyset$. If the ball $B(\mathbf{z}, r)$ intersects only one element in $\mathcal{F}_{k}$ for any $k>0$, then by \eqref{32301} we have $\mu(B(\mathbf{z}, r))=0$. So we can take an integer $k$ such that the ball $B(\mathbf{z}, r)$ intersects with only one element in $\mathcal{F}_{k-1}$, and at least with two elements in $\mathcal{F}_{k}$. Let $Q_{k-1}$ be the unique element in $\mathcal{F}_{k-1}$ that intersects the ball $B(\mathbf{z}, r)$.

If $r \ge \beta_2^{-h_{k-1}}$, then
\begin{equation*}
\mu(B(\mathbf{z}, r)) \leq \mu\left(Q_{k-1}\right) \ll{\beta_2}^{-sh_{k-1}} \leq r^s.
\end{equation*}
Without loss of generality, assume that $r \leq  \beta_2^{-h_{k-1}}.$ Combining  Lemma \ref{lengap} \eqref{gap0} with the fact that the ball $B(\mathbf{z}, r)$ intersects at least two elements in $\mathcal{F}_{k},$ we obtain
\begin{equation*}
\frac{1}{\bt_2^{h_k}} \asymp \frac{1}{\bt_1^{\tilde{h}_k}}\le r<\frac{1}{\beta_2^{h_{k-1}}},
\end{equation*}
which implies $k\ge k_0.$ Based on the above discussion, we consider the following three cases.

{\em {Case 1.}}
If $\beta_1^{1+v}>\beta_2$, this corresponds to {\em {Case 1}} of the upper bound. In this case we have
\begin{equation*}
\beta_2^{-h_{k-1}} \geq \beta_1^{-l_k} \geq \beta_2^{-l_k} > \beta_1^{-h_k} \geq \beta_2^{-h_k}.
\end{equation*}
We consider four subcases as follows.

{\em {Case 1.1.}} $\beta_2^{-h_{k-1}} \geq r \geq \beta_1^{-l_k}$.

We estimate the number $T$ of basic squares $Q_k\in \F_k(Q_{k-1})$ that intersect $B(\mathbf{z}, r).$ By the volume calculation and Lemma \ref{lengap} \eqref{gap0}, we have
\begin{align*}
  &\sharp\left\{R_k \in \mR_k\left(Q_{k-1}\right): R_k \cap B(\mathbf{z}, r) \neq \emptyset\right\} \\
\leq &\left(2 \cdot r \cdot \beta_1^{l_k}+2\right)\left(2 \cdot r \cdot \beta_2^{l_k}+2\right).
\end{align*} Thus,
\begin{equation*}
  T\le \sharp\left\{R_k \in \mR_k\left(Q_{k-1}\right): R_k \cap B(\mathbf{z}, r) \neq \emptyset\right\}\cdot \sharp  \mQ_k(R_{k}).
\end{equation*}

Hence we obtain
\begin{equation*}
\begin{aligned}
\mu(B(\mathbf{z}, r))
&\leq T \cdot \mu(Q_k)\ll r^2 \cdot \beta_1^{l_k}\beta_2^{l_k} \cdot \mu(R_k)\\
&= r^2 \cdot \beta_1^{l_k}\cdot \beta_2^{l_k} \cdot \beta_1^{-[l_k - \sum_{i=1}^{k-1}(h_i - l_i)-2N-1-\sum_{i=1}^{k-1}(\tilde{t}_i+\bar{t}_i+2)(2 N+1)]} \\
&\times \beta_2^{-[l_k - \sum_{i=1}^{k-1}(h_i - l_i)-2N-1-\sum_{i=1}^{k-1}(t_i+1)(2 N+1)]} .
\end{aligned}
\end{equation*}
Combining \eqref{eq31101}, if $t$ satisfies
\begin{equation}\label{32402}
t\leq\liminf _{k \rightarrow \infty} \left[2+\frac{l_k \log \beta_1 + l_k \log \beta_2 - [l_k - \sum_{i=1}^{k-1}(h_i - l_i)](\log \beta_1 + \log \beta_2)}{\log r}\right]=:t_1,
\end{equation}
then for any sufficiently small $r$, we get $\mu(B(\mathbf{z}, r)) \ll  r^t$ for every $t\leq t_1$.

According to the Stolz-Ces\`{a}ro Theorem, we can deduce that
\begin{equation}\label{32401}
\begin{split}
\lim _{k \rightarrow \infty} \frac{l_k-\sum_{i=1}^{k-1}\left(h_i-l_i\right)}{l_k}
& =\lim _{k \rightarrow \infty} \frac{l_k-h_{k-1}}{l_k-l_{k-1}} \\
& =\frac{v-\hat{v}-\hat{v} v}{v-\hat{v}},
\end{split}
\end{equation}
where the final equality follows from \eqref{e2} and \eqref{e3}.

Notice that by applying \eqref{32401}, we have
\begin{equation*}
\begin{aligned}
t_1&\leq \liminf _{k \rightarrow \infty}\left[ 2 + \frac{l_k \log \beta_1 + l_k \log \beta_2 - [l_k - \sum_{i=1}^{k-1}(h_i -l_i)](\log \beta_1 + \log \beta_2)}{-{h_{k-1}} \log \beta_2}\right] \\
&=2-\left(1+\log _{\beta_2} \beta_1\right)\frac{v^2}{(1+v)(v-\hat{v})},
\end{aligned}
\end{equation*}
that is
\begin{equation*}
 t\leq 2-\left(1+\log _{\beta_2} \beta_1\right)\frac{v^2}{(1+v)(v-\hat{v})}.
\end{equation*}
Hence, \eqref{32402} holds.

{\em {Case 1.2.}}   $\beta_1^{-l_k} > r \geq \beta_2^{-l_k}$.

Using the same method as in {\em Case 1.1}, we have
\begin{equation*}
\begin{aligned}
\mu(B(\mathbf{z}, r))
&\leq T \cdot \mu(Q_k)\ll r\beta_2^{l_k} \cdot \mu(R_k) \\
&\ll r^2 \cdot \beta_2^{l_k} \cdot \beta_1^{-[l_k - \sum_{i=1}^{k-1}(h_i - l_i)-2N-1-\sum_{i=1}^{k-1}(\tilde{t}_i+\bar{t}_i+2)(2 N+1)]} \\
&\times \beta_2^{-[l_k - \sum_{i=1}^{k-1}(h_i - l_i)-2N-1-\sum_{i=1}^{k-1}(t_i+1)(2 N+1)]}.
\end{aligned}
\end{equation*}
Combining \eqref{eq31101}, if $t$ satisfies
\begin{equation}\label{c121}
\begin{aligned}
 t&\leq \liminf _{k \rightarrow \infty} \left[1 + \frac{l_k \log \beta_2 - [n_k - \sum_{i=1}^{k-1}(m_i - n_i)](\log \beta_1 + \log \beta_2)}{\log r}\right] \\
&= \min \left\{\liminf _{k \rightarrow \infty} \left[1 + \frac{l_k \log \beta_2 - \left[n_k - \sum_{i=1}^{k-1}(m_i - n_i)\right](\log \beta_1 + \log \beta_2)}{-l_k \log \beta_1}\right],\right. \\
&\left.\quad  \liminf _{k \rightarrow \infty} \left[1 + \frac{l_k \log \beta_2 - \left[n_k - \sum_{i=1}^{k-1}(m_i - n_i)\right](\log \beta_1 + \log \beta_2)}{-l_k \log \beta_2}\right]\right\} \\
&=:t_2,
\end{aligned}
\end{equation}
for any sufficiently small $r$, it follows that $\mu(B(\mathbf{z}, r)) \ll r^t$ for every $t\leq t_2$. Then, by combining \eqref{32401}, one has
\begin{equation*}
\begin{aligned}
&\liminf _{k \rightarrow \infty} \left[1 + \frac{l_k \log \beta_2 - \left[n_k - \sum_{i=1}^{k-1}(m_i - n_i)\right](\log \beta_1 + \log \beta_2)}{-l_k \log \beta_1}\right]\\
= & 1-\log _{\beta_1}{\beta_2} + (1+\log _{\beta_1}{\beta_2}) \frac{v- \hat{v}-v\hat{v}}{v - \hat{v}},
\end{aligned}
\end{equation*}
and
\begin{equation*}
\begin{aligned}
&\liminf _{k \rightarrow \infty} \left[1 + \frac{l_k \log \beta_2 - \left[n_k - \sum_{i=1}^{k-1}(m_i - n_i)\right](\log \beta_1 + \log \beta_2)}{-l_k \log \beta_2}\right]\\
= &(1+\log _{\beta_2}{\beta_1}) \frac{v- \hat{v}-v\hat{v}}{v - \hat{v}}.
\end{aligned}
\end{equation*}
That is to say,
\begin{equation*}
t\leq \min \left\{1-\log _{\beta_1}{\beta_2} + (1+\log _{\beta_1}{\beta_2}) \frac{v- \hat{v}-v\hat{v}}{v - \hat{v}}, ~(1+\log _{\beta_2}{\beta_1})\frac{v- \hat{v}-v\hat{v}}{v - \hat{v}} \right\}.
\end{equation*}
Therefore, \eqref{c121} holds.

{\em {Case 1.3.}} $\beta_2^{-l_k} > r > \beta_1^{-h_k}$.

In this situation, the ball $B(\mathbf{z}, r)$ intersects exactly one rectangle. Hence we obtain
\begin{equation*}
\begin{aligned}
\mu(B(\mathbf{z}, r))
&\leq T \cdot \mu(Q_k)\\
&\leq \beta_1^{-[l_k - \sum_{i=1}^{k-1}(h_i - l_i)-2N-1-\sum_{i=1}^{k-1}(\tilde{t}_i+\bar{t}_i+2)(2 N+1)]} \\
&\times \beta_2^{-[l_k - \sum_{i=1}^{k-1}(h_i - l_i)-2N-1-\sum_{i=1}^{k-1}(t_i+1)(2 N+1)]}.
\end{aligned}
\end{equation*}
Combining \eqref{eq31101}, if $t$ satisfies
\begin{equation}\label{c131}
\begin{aligned}
t&\leq \liminf _{k \rightarrow \infty} \frac{-\left[n_k - \sum_{i=1}^{k-1}(m_i - n_i)\right](\log \beta_1 + \log \beta_2)}{\log r} \\
&= \min \left\{\liminf _{k \rightarrow \infty}\frac{\left[n_k - \sum_{i=1}^{k-1}(m_i - n_i)\right](\log \beta_1 + \log \beta_2)}{l_k \log \beta_2},\right. \\
&\left.\quad \liminf _{k \rightarrow \infty}\frac{\left[n_k - \sum_{i=1}^{k-1}(m_i - n_i)\right](\log \beta_1 + \log \beta_2)}{h_k \log \beta_1}\right\}=:t_3,
\end{aligned}
\end{equation}
for any sufficiently small $r$, it can be shown that $\mu(B(\mathbf{z}, r)) \leq r^t$ holds for every $t\leq t_3$.
Using \eqref{32302} and \eqref{32401}, we obtain
\begin{equation*}
\liminf _{k \rightarrow \infty}\frac{\left[n_k - \sum_{i=1}^{k-1}(m_i - n_i)\right](\log \beta_1 + \log \beta_2)}{l_k \log \beta_2}
= (1+\log _{\beta_2} \beta_1) \frac{v-\hat{v}-v \hat{v}}{v - \hat{v}}
\end{equation*}
and
\begin{equation*}
\liminf _{k \rightarrow \infty}\frac{\left[n_k - \sum_{i=1}^{k-1}(m_i - n_i)\right](\log \beta_1 + \log \beta_2)}{h_k \log \beta_1}
= (1+\log _{\beta_1} \beta_2)\cdot\frac{v-\hat{v}-v \hat{v}}{(1 + v)(v - \hat{v})}.
\end{equation*}
Thus, we have
\begin{equation*}
t\leq \min \left\{(1+\log _{\beta_2} \beta_1)\frac{v-\hat{v}-v \hat{v}}{v - \hat{v}},(1+\log _{\beta_1} \beta_2)\cdot\frac{v-\hat{v}-v \hat{v}}{(1 + v)(v - \hat{v})}\right\} .
\end{equation*}
Hence, \eqref{c131} holds.

{\em {Case 1.4.}}  \( \beta_1^{-h_k} \geq r \geq \beta_2^{-h_k} \).

In this situation, we deduce that
\begin{equation*}
\begin{aligned}
\mu(B(\mathbf{z}, r))
&\leq T \cdot \mu(Q_k) \leq \frac{r}{\beta_2^{-h_k}} \cdot\mu(R_k) \cdot \frac{\beta_2^{-h_k}}{\beta_1^{-h_k}}\\
&=r\cdot\beta_1^{h_k} \beta_1^{-[l_k - \sum_{i=1}^{k-1}(h_i - l_i)-2N-1-\sum_{i=1}^{k-1}(\tilde{t}_i+\bar{t}_i+2)(2 N+1)]} \\
&\times \beta_2^{-[l_k - \sum_{i=1}^{k-1}(h_i - l_i)-2N-1-\sum_{i=1}^{k-1}(t_i+1)(2 N+1)]}.
\end{aligned}
\end{equation*}
Combining \eqref{eq31101}, if $t$ satisfies
\begin{equation}\label{c141}
\begin{aligned}
t
&\leq \min \left\{\liminf _{k \rightarrow \infty} \left[1 + \frac{h_k \log \beta_1 - [n_k - \sum_{i=1}^{k-1}(m_i - n_i)](\log \beta_1 + \log \beta_2)}{-h_k \log \beta_1}\right],\right. \\
&\left. \quad  \liminf _{k \rightarrow \infty} \left[1 + \frac{h_k \log \beta_1 - \left[n_k - \sum_{i=1}^{k-1}(m_i - n_i)\right](\log \beta_1 + \log \beta_2)}{-h_k \log \beta_2}\right]\right\} \\
&=:t_4,
\end{aligned}
\end{equation}
then for any sufficiently small $r$, we have $\mu(B(\mathbf{z}, r)) \leq r^t$ for every $t\leq t_4$. By applying \eqref{32302}, we can obtain
\begin{equation*}
\begin{aligned}
&\liminf _{k \rightarrow \infty} \left[1 + \frac{h_k \log \beta_1 - [n_k - \sum_{i=1}^{k-1}(m_i - n_i)](\log \beta_1 + \log \beta_2)}{-h_k \log \beta_1}\right]\\
= & (1+\log _{\beta_1} \beta_2)\frac{v-\hat{v}-v \hat{v}}{(1+v)(v - \hat{v})},
\end{aligned}
\end{equation*}
and
\begin{equation*}
\begin{aligned}
&\liminf _{k \rightarrow \infty} \left[1 + \frac{h_k \log \beta_1 - \left[n_k - \sum_{i=1}^{k-1}(m_i - n_i)\right](\log \beta_1 + \log \beta_2)}{-h_k \log \beta_2}\right]\\
=& 1 - \log _{\beta_2} \beta_1 + (1+\log _{\beta_2} \beta_1)\frac{v- \hat{v}-v \hat{v}}{(1 + v)(v - \hat{v})}.
\end{aligned}
\end{equation*}
Equivalently, if
\begin{equation*}
t\leq \min \left\{(1+\log _{\beta_1} \beta_2)\frac{v-\hat{v}-v \hat{v}}{(1+v)(v - \hat{v})}, 1 - \log _{\beta_2} \beta_1 + (1+\log _{\beta_2} \beta_1)\frac{v- \hat{v}-v \hat{v}}{(1 + v)(v - \hat{v})}\right\},
\end{equation*}
then \eqref{c141} holds.

From the four cases presented above, we deduce that for any $t\leq \min\{t_1,t_2,t_{3},t_{4}\}$,
\begin{equation*}
\mu(B(\mathbf{z}, r)) \ll r^t.
\end{equation*}
Thus, applying Lemma \ref{lem3003}, we get
\begin{equation*}
\hdim E_{\beta_1, \beta_2}(\hat{v}, v)\geq t.
\end{equation*}
Then, by the arbitrariness of $t$, we have
\begin{equation*}
\hdim E_{\beta_1, \beta_2}(\hat{v}, v)\geq \min\{t_1,t_2,t_3,t_{4}\}.
\end{equation*}
Using the inequality $\hat{v} \leq \frac{v}{1+v}$, we obtain
\begin{equation*}
\begin{aligned}
&\min\{t_1,t_2,t_{3},t_{4}\}\\
=& \min \left\{\left(1+\log _{\beta_1} {\beta_2}\right)\frac{v-\hat{v}-v \hat{v}}{(1+v)(v - \hat{v})},1 - \log _{\beta_2} {\beta_1} + \left(1+\log _{\beta_2} {\beta_1}\right) \frac{v-\hat{v}-v \hat{v}}{(1 + v)(v - \hat{v})}\right\}.
\end{aligned}
\end{equation*}
Hence, when $ \beta_1^{1+v} > \beta_2$, we conclude that
\begin{equation}\label{c1}
\begin{aligned}
&\hdim E_{\beta_1, \beta_2}(\hat{v}, v) \\
\geq& \min \left\{\frac{(1+\log _{\beta_1} {\beta_2})(v-\hat{v}-v \hat{v})}{(1+v)(v - \hat{v})},1 - \log _{\beta_2} {\beta_1} + \frac{(1+\log _{\beta_2} {\beta_1})(v-\hat{v}-v \hat{v})}{(1 + v)(v - \hat{v})}\right\}.
\end{aligned}
\end{equation}

{\em {Case 2.}}
If $\beta_1^{1+v}< \beta_2$, this situation is denoted as {\em {Case 2}} for the upper bound. Taking into account the construction methods of the Cantor subset and the sequences $\left\{h_k\right\}$ and $\left\{l_k\right\}$, it follows that $\beta_1^{-h_k}> \beta_2^{-l_k}$. In this case, we have
\begin{equation*}
\beta_2^{-h_{k-1}} \geq \beta_1^{-l_k} \geq \beta_1^{-h_k} >  \beta_2^{-l_k} \geq  \beta_2^{-h_k}.
\end{equation*}
There are four subcases, which are listed as follows.

{\em {Case 2.1.}} $\beta_2^{-h_{k-1}} \geq r \geq \beta_1^{-l_k} $.

Similarly to {\em {Case 1.1}}, we also have
\begin{equation*}
\begin{aligned}
\mu(B(\mathbf{z}, r)) & \leq T \cdot \mu\left(Q_k\right) \\
& \ll r^2\cdot \beta_1^{l_k}\cdot\beta_2^{l_k}\cdot \beta_1^{-[l_k - \sum_{i=1}^{k-1}(h_i - l_i)-2N-1-\sum_{i=1}^{k-1}(\tilde{t}_i+\bar{t}_i+2)(2 N+1)]} \\
&\times \beta_2^{-[l_k - \sum_{i=1}^{k-1}(h_i - l_i)-2N-1-\sum_{i=1}^{k-1}(t_i+1)(2 N+1)]},
\end{aligned}
\end{equation*}
and thus, for any sufficiently small $r$, we obtain $\mu(B(\mathbf{z}, r))\ll r^t$ for every $t\leq t_1$.

{\em {Case 2.2.}} $ \beta_1^{-l_k} > r \geq \beta_1^{-h_k} $.

Now, we can obtain
\begin{equation*}
\begin{aligned}
\mu(B(\mathbf{z}, r))
&\ll r \cdot\beta_2^{l_k} \cdot\beta_1^{-[l_k - \sum_{i=1}^{k-1}(h_i - l_i)-2N-1-\sum_{i=1}^{k-1}(\tilde{t}_i+\bar{t}_i+2)(2 N+1)]} \\
&\times \beta_2^{-[l_k - \sum_{i=1}^{k-1}(h_i - l_i)-2N-1-\sum_{i=1}^{k-1}(t_i+1)(2 N+1)]} .
\end{aligned}
\end{equation*}
Combining \eqref{eq31101}, if $t$ satisfies
\begin{equation}\label{c221}
\begin{aligned}
t &\leq \min \left\{\liminf _{k \rightarrow \infty} \left[1 + \frac{l_k \log \beta_2 - [n_k - \sum_{i=1}^{k-1}(m_i - n_i)](\log \beta_1 + \log \beta_2)}{-l_k \log \beta_1}\right],\right. \\
&\left.\quad \liminf _{k \rightarrow \infty} \left[ 1 + \frac{l_k \log \beta_2 - [n_k - \sum_{i=1}^{k-1}(m_i - n_i)](\log \beta_1 + \log \beta_2)}{-h_k \log \beta_1}\right]\right\} \\
&=:t_5,
\end{aligned}
\end{equation}
then for any sufficiently small $r$, we get $\mu(B(\mathbf{z}, r)) \ll r^t$ for every $t \leq t_5$. Observe that, by \eqref{32302} and \eqref{32401},
\begin{equation*}
\begin{aligned}
&\liminf _{k \rightarrow \infty} \left[1 + \frac{l_k \log \beta_2 - [n_k - \sum_{i=1}^{k-1}(m_i - n_i)](\log \beta_1 + \log \beta_2)}{-l_k \log \beta_1}\right]\\
=&1-\log _{\beta_1}{\beta_2} + (1+\log _{\beta_1}{\beta_2}) \frac{v- \hat{v}-v\hat{v}}{v - \hat{v}}
\end{aligned}
\end{equation*}
and
\begin{equation*}
\begin{aligned}
&\liminf _{k \rightarrow \infty} \left[ 1 + \frac{l_k \log \beta_2 - [n_k - \sum_{i=1}^{k-1}(m_i - n_i)](\log \beta_1 + \log \beta_2)}{-h_k \log \beta_1}\right]\\
=&1 - \frac{1}{(1 + v)} \log _{\beta_1} \beta_2 + \left(1+\log _{\beta_1} \beta_2\right) \frac{v-\hat{v}-v \hat{v}}{(1+v)(v-\hat{v})} .
\end{aligned}
\end{equation*}
Thus, we conclude that
\begin{equation*}
\begin{aligned}
t\leq& \min \left\{1-\log _{\beta_1}{\beta_2} +(1+\log _{\beta_1}{\beta_2}) \frac{v- \hat{v}-v\hat{v}}{v - \hat{v}},\right. \\
&\left.1 - \frac{1}{(1 + v)} \log _{\beta_1} \beta_2 + \left(1+\log _{\beta_1} \beta_2\right) \frac{v-\hat{v}-v \hat{v}}{(1+v)(v-\hat{v})}\right\},
\end{aligned}
\end{equation*}
and then \eqref{c221} holds.

{\em {Case 2.3.}} $ \beta_1^{-h_k} > r > \beta_2^{-l_k} $.

In this situation, we first estimate
\begin{equation*}
\begin{aligned}
\mu(B(\mathbf{z}, r)) &\ll  r^2\beta_2^{l_k} \cdot \beta_1^{h_k} \cdot \beta_1^{-[l_k - \sum_{i=1}^{k-1}(h_i - l_i)-2N-1-\sum_{i=1}^{k-1}(\tilde{t}_i+\bar{t}_i+2)(2 N+1)]} \\
&\times \beta_2^{-[l_k - \sum_{i=1}^{k-1}(h_i - l_i)-2N-1-\sum_{i=1}^{k-1}(t_i+1)(2 N+1)]} .
\end{aligned}
\end{equation*}
Combining \eqref{eq31101}, if $t$ satisfies
\begin{equation}\label{c231}
\begin{aligned}
t &\leq\min \left\{\liminf _{k \rightarrow \infty} \left[2 + \frac{l_k \log \beta_2 - [n_k - \sum_{i=1}^{k-1}(m_i - n_i)](\log \beta_1 + \log \beta_2) + h_k \log \beta_1}{-h_k \log \beta_1}\right],\right. \\
&\left.\quad \liminf _{k \rightarrow \infty} \left[2 + \frac{l_k \log \beta_2 - [n_k - \sum_{i=1}^{k-1}(m_i - n_i)](\log \beta_1 + \log \beta_2) + h_k \log \beta_1}{-l_k \log \beta_2}\right]\right\} \\
&=:t_6,
\end{aligned}
\end{equation}
and consequently, for any sufficiently small $r$, it can be shown that $\mu(B(\mathbf{z}, r)) \ll r^t$ holds for every $t\leq t_6$.
From \eqref{32302} and \eqref{32401},
\begin{equation*}
\begin{aligned}
&\liminf _{k \rightarrow \infty} \left[2 + \frac{l_k \log \beta_2 - [n_k - \sum_{i=1}^{k-1}(m_i - n_i)](\log \beta_1 + \log \beta_2) + h_k \log \beta_1}{-h_k \log \beta_1}\right]\\
=&1 - \frac{1}{(1 + v)} \log _{\beta_1} \beta_2 + \left(1+\log _{\beta_1} \beta_2\right) \frac{v-\hat{v}-v \hat{v}}{(1+v)(v-\hat{v})}
\end{aligned}
\end{equation*}
and
\begin{equation*}
\begin{aligned}
&\liminf _{k \rightarrow \infty} \left[2 + \frac{l_k \log \beta_2 - [n_k - \sum_{i=1}^{k-1}(m_i - n_i)](\log \beta_1 + \log \beta_2) + h_k \log \beta_1}{-l_k \log \beta_2}\right]\\
=&1 - (1 + v) \log _{\beta_2} \beta_1 + \left(1+\log _{\beta_2} \beta_1\right)  \frac{v- \hat{v}-v\hat{v}}{v - \hat{v}}.
\end{aligned}
\end{equation*}
That is,
\begin{equation*}
\begin{aligned}
t &\leq\min \left\{1 - \frac{1}{(1 + v)} \log _{\beta_1} \beta_2 + \left(1+\log _{\beta_1} \beta_2\right) \frac{v-\hat{v}-v \hat{v}}{(1+v)(v-\hat{v})},\right. \\
&\left.\quad 1 - (1 + v) \log _{\beta_2} \beta_1 + \left(1+\log _{\beta_2} \beta_1\right)  \frac{v- \hat{v}-v\hat{v}}{v - \hat{v}}\right\},
\end{aligned}
\end{equation*}
in this case, \eqref{c231} holds.

{\em {Case 2.4.}} $ \beta_2^{-l_k} \geq r \geq \beta_2^{-h_k} $.

In this situation, we first calculate
\begin{equation*}
\begin{aligned}
\mu(B(\mathbf{z}, r))
&\leq T \cdot \mu(Q_k)\\
&\leq \frac{r}{\beta_2^{-h_k}} \cdot\mu(R_k) \cdot \frac{\beta_2^{-h_k}}{\beta_1^{-h_k}}\\
&=r\cdot\beta_1^{h_k} \cdot \beta_1^{-[l_k - \sum_{i=1}^{k-1}(h_i - l_i)-2N-1-\sum_{i=1}^{k-1}(\tilde{t}_i+\bar{t}_i+2)(2 N+1)]} \\
&\times \beta_2^{-[l_k - \sum_{i=1}^{k-1}(h_i - l_i)-2N-1-\sum_{i=1}^{k-1}(t_i+1)(2 N+1)]} .
\end{aligned}
\end{equation*}
Combining \eqref{eq31101}, if $t$ satisfies
\begin{equation}\label{c241}
\begin{aligned}
t &\leq \min \left\{\liminf _{k \rightarrow \infty} \left[1 + \frac{h_k \log \beta_1 - [n_k - \sum_{i=1}^{k-1}(m_i - n_i)](\log \beta_1 + \log \beta_2)}{-l_k \log \beta_2}\right],\right. \\
&\left.\quad  \liminf _{k \rightarrow \infty} \left[1 + \frac{h_k \log \beta_1 - [n_k - \sum_{i=1}^{k-1}(m_i - n_i)](\log \beta_1 + \log \beta_2)}{-h_k \log \beta_2}\right]\right\} \\
&=:t_7,
\end{aligned}
\end{equation}
then, for any sufficiently small $r$, we obtain $\mu(B(\mathbf{z}, r)) \leq r^t$ for every $t\leq t_7$.
By \eqref{32302} and \eqref{32401},
\begin{equation*}
\begin{aligned}
&\liminf _{k \rightarrow \infty} \left[1 + \frac{h_k \log \beta_1 - [n_k - \sum_{i=1}^{k-1}(m_i - n_i)](\log \beta_1 + \log \beta_2)}{-l_k \log \beta_2}\right]\\
=&1 - (1 + v) \log _{\beta_2} \beta_1 + \left(1+\log _{\beta_2} \beta_1\right)  \frac{v - \hat{v}-v \hat{v}}{v - \hat{v}}
\end{aligned}
\end{equation*}
and
\begin{equation*}
\begin{aligned}
&\liminf _{k \rightarrow \infty} \left[1 + \frac{h_k \log \beta_1 - [n_k - \sum_{i=1}^{k-1}(m_i - n_i)](\log \beta_1 + \log \beta_2)}{-h_k \log \beta_2}\right]\\
=&1 - \log _{\beta_2} \beta_1 + \left(1+\log _{\beta_2} \beta_1\right)\cdot \frac{v-\hat{v}-v \hat{v}}{(1+v) (v - \hat{v})} .
\end{aligned}
\end{equation*}
So we obtain
\begin{equation*}
\begin{aligned}
t &\leq\min \left\{1 - (1 + v) \log _{\beta_2} \beta_1 + \left(1+\log _{\beta_2} \beta_1\right)  \left(1 - \frac{v \hat{v}}{v - \hat{v}}\right),\right. \\
&\left.\quad 1 - \log _{\beta_2} \beta_1 + \left(1+\log _{\beta_2} \beta_1\right)\cdot \frac{v-\hat{v}-v \hat{v}}{(1+v) (v - \hat{v})}\right\}.
\end{aligned}
\end{equation*}
Therefore, \eqref{c241} holds.

From the four cases above, we conclude that
\begin{equation*}
\mu(B(\mathbf{z}, r)) \ll r^t \text{~for any~} t\leq \min\{t_1,t_{5},t_6,t_{7}\}.
\end{equation*}
Therefore, by Lemma \ref{lem3003}, we get
\begin{equation*}
\hdim E_{\beta_1, \beta_2}(\hat{v}, v)\geq t.
\end{equation*}
Then, by the arbitrariness of $t$, we have
\begin{equation*}
\hdim E_{\beta_1, \beta_2}(\hat{v}, v)\geq \min\{t_1,t_{5},t_6,t_{7}\}.
\end{equation*}
By $\hat{v} \leq \frac{v}{1+v}$, we obtain
\begin{equation*}
\begin{aligned}
&\min\{t_1,t_{5},t_6,t_{7}\}\\
= &\min  \left\{1-\frac{1}{(1+v)}\log _{\beta_1} {\beta_2}+\left(1+\log _{\beta_1} {\beta_2}\right)\cdot \frac{v-\hat{v}-v \hat{v}}{(1+v)(v-\hat{v})},\right. \\
&\left. 1 - \log _{\beta_2} {\beta_1} + \left(1+\log _{\beta_2} {\beta_1}\right)\cdot\frac{v-\hat{v}-v \hat{v}}{(1 + v)(v - \hat{v})}\right\}.
\end{aligned}
\end{equation*}
Hence, when $ \beta_1^{-h_k} >\beta_2^{-l_k}$, we conclude that
\begin{equation}\label{c2}
\begin{aligned}
&\hdim E_{\beta_1, \beta_2}(\hat{v}, v) \\
\geq& \min \left\{1-\frac{1}{(1+v)}\log _{\beta_1} {\beta_2}+ \frac{(1+\log _{\beta_1} {\beta_2})(v-\hat{v}-v \hat{v})}{(1+v)(v-\hat{v})},\right. \\
&\left. 1 - \log _{\beta_2} {\beta_1} + \frac{(1+\log _{\beta_2} {\beta_1})(v-\hat{v}-v \hat{v})}{(1 + v)(v - \hat{v})}\right\}.
\end{aligned}
\end{equation}

{\em {Case 3.}}
Let $\varepsilon>0$. When $\beta_1^{1+v}=\beta_2$, we consider two cases: $\beta_1^{-h_k} \leq \beta_2^{-l_k(1-\varepsilon)}$ and $\beta_1^{-h_k} \geq \beta_2^{-l_k(1+\varepsilon)}$.

{\em {Case 3.1.}} $\beta_1^{-h_k} \leq \beta_2^{-l_k(1-\varepsilon)}$. This condition is similar to that in {\em {Case 1}}, and we only need to change $\beta_2^{-l_k}$ to $\beta_2^{-l_k(1-\varepsilon)}$ in the arguments for {\em {Case 1}}. By the same reasiong, we have
\begin{equation*}
\mu(B(\mathbf{z}, r)) \ll r^t \text{~ for any~} t \leq \min \left\{t_1, t_2^{\prime}, t_3^{\prime}, t_4\right\},
\end{equation*}
where $t_1$ and $t_4$ are the same as in {\em {Case 1}},
\begin{equation*}
t_2^{\prime} := \min \left\{1-(1-\varepsilon)\log _{\beta_1} \beta_2+\left(1+\log _{\beta_1} \beta_2\right) \frac{v-\hat{v}-v \hat{v}}{v-\hat{v}},\frac{(1+\log _{\beta_2} \beta_1)(v-\hat{v}-v\hat{v})}{(1-\varepsilon)(v-\hat{v})}\right\},
\end{equation*}
and
\begin{equation*}
t_3^{\prime} := \min \left\{\left(1+\log _{\beta_1} \beta_2\right)\frac{v-\hat{v}-v\hat{v}}{(1+v)(v-\hat{v})},\frac{(1+\log _{\beta_2} \beta_1)(v-\hat{v}-v\hat{v})}{(1-\varepsilon)(v-\hat{v})}\right\}.
\end{equation*}
Hence, by Lemma \ref{lem3003}, it follows that
\begin{equation*}
\hdim E_{\beta_1, \beta_2}(\hat{v}, v)\geq t.
\end{equation*}
By the arbitrariness of $t$, we have
\begin{equation*}
\hdim E_{\beta_1, \beta_2}(\hat{v}, v)\geq \min \left\{t_1, t_2^{\prime}, t_3^{\prime}, t_4\right\}.
\end{equation*}
Note that $\hat{v} \leq \frac{v}{1+v}$ under the condition $\beta_1^{-h_k} \leq \beta_2^{-l_k(1-\varepsilon)}$. Letting $\varepsilon \rightarrow 0$, we obtain
\begin{equation*}
\begin{aligned}
& \hdim E_{\beta_1, \beta_2}(\hat{v}, v) \\
\geq & \min \left\{\frac{\left(1+\log _{\beta_1} \beta_2\right)(v-\hat{v}-v \hat{v})}{(1+v)(v-\hat{v})}, 1-\log _{\beta_2} \beta_1+\frac{\left(1+\log _{\beta_2} \beta_1\right)(v-\hat{v}-v \hat{v})}{(1+v)(v-\hat{v})}\right\}.
\end{aligned}
\end{equation*}
This is equivalent to the result of \eqref{c1}.

{\em {Case 3.2.}} $\beta_1^{-h_k} \geq \beta_2^{-l_k(1+\varepsilon)}$. This situation is comparable to that in {\em {Case 2}}. We simply replace $\beta_2^{-l_k}$ with $\beta_2^{-l_k(1+\varepsilon)}$ in the arguments for {\em {Case 2}} and thus obtain
\begin{equation*}
\mu(B(\mathbf{z}, r)) \ll r^t \text { for any } t \leq \min \left\{t_1, t_5,  t_6^{\prime}, t_7^{\prime}\right\},
\end{equation*}
where $t_1$ and $t_5$ take the same values as in {\em {Case 2}},
\begin{equation*}
\begin{aligned}
t_6^{\prime} := & \min \left\{1-\frac{1+\varepsilon}{(1+v)} \log _{\beta_1} \beta_2+\left(1+\log _{\beta_1} \beta_2\right) \frac{v-\hat{v}-v \hat{v}}{(1+v)(v-\hat{v})},\right. \\
& \left.1-\frac{1+v}{1+\varepsilon}\log _{\beta_2} \beta_1+ \frac{(1+\log _{\beta_2} \beta_1)(v-\hat{v}-v \hat{v})}{(1+\varepsilon)(v-\hat{v})}\right\},
\end{aligned}
\end{equation*}
and
\begin{equation*}
\begin{aligned}
t_7^{\prime} := & \min \left\{1-\frac{1+v}{1+\varepsilon} \log _{\beta_2} \beta_1+\frac{(1+\log _{\beta_2} \beta_1)(v-\hat{v}-v \hat{v})}{(1+\varepsilon)(v-\hat{v})}, \right. \\
& \left.1-\log _{\beta_2} \beta_1+\left(1+\log _{\beta_2} \beta_1\right)  \frac{v-\hat{v}-v \hat{v}}{(1+v)(v-\hat{v})}\right\}.
\end{aligned}
\end{equation*}
Thus, from Lemma \ref{lem3003}, we obtain
\begin{equation*}
\hdim E_{\beta_1, \beta_2}(\hat{v}, v)\geq t.
\end{equation*}
By the arbitrariness of $t$, it follows that
\begin{equation*}
\hdim E_{\beta_1, \beta_2}(\hat{v}, v)\geq \min \left\{t_1, t_5,  t_6^{\prime}, t_7^{\prime}\right\}.
\end{equation*}
Applying the inequality $\hat{v} \leq \frac{v}{1+v}$ under $\beta_1^{-h_k} \geq \beta_2^{-l_k(1-\varepsilon)}$ and taking the the limit as $\varepsilon \rightarrow 0$, we conclude that
\begin{equation*}
\begin{aligned}
& \hdim E_{\beta_1, \beta_2}(\hat{v}, v) \\
\geq & \min \left\{1-\frac{1}{(1+v)} \log _{\beta_1} \beta_2+\frac{\left(1+\log _{\beta_1} \beta_2\right)(v-\hat{v}-v \hat{v})}{(1+v)(v-\hat{v})},\right. \\
& \left.1-\log _{\beta_2} \beta_1+\frac{\left(1+\log _{\beta_2} \beta_1\right)(v-\hat{v}-v \hat{v})}{(1+v)(v-\hat{v})}\right\} .
\end{aligned}
\end{equation*}
This is the same as the result of \eqref{c2}.

According to \eqref{c1} and \eqref{c2}, we have thus established the lower bound of the Hausdorff dimension of $E_{\beta_1, \beta_2}(\hat{v}, v)$.

\section{\bf Proof of Theorem \ref{thm2}}\
In this section, we provide the proof of Theorem \ref{thm2}, which is a corollary of Theorem \ref{thm3}.

Firstly, when $\hat{v}=v=0$, we have
\begin{equation*}
\mathcal{L}^2\left(E_{\beta_1, \beta_2}\left(0, 0\right)\right)=1.
\end{equation*}
Moreover, since
\begin{equation*}
E_{\beta_1, \beta_2}\left(0, 0\right) \subset \left\{(x, y) \in [0,1]^2: \hat{v}_{\beta_1, \beta_2}(x, y)=0\right\},
\end{equation*}
we obtain
\begin{equation*}
\hdim \left\{(x, y) \in[0,1]^2: \hat{v}_{\beta_1, \beta_2}(x, y)=0\right\}=2.
\end{equation*}

Next, when $\hat{v}>1$, it follows from (\ref{tm1.6}) that
\begin{equation*}
\hdim \left\{(x, y) \in [0,1]^2: \hat{v}_{\beta_1, \beta_2}(x, y) \geq \hat{v}\right\}=0.
\end{equation*}

It remains to consider the case when $0<\hat{v}\leq1$. Observe that the inclusion \( E_{\beta_1, \beta_2}(\hat{v}, v) \subset U_{\beta_1, \beta_2}(\hat{v}) \) holds for any \( 0 \le v \le \infty \). Consequently,
\begin{equation*}
\hdim U_{\bt_1,\bt_2}(\hv) \ge \sup\limits_{0\le v \le \infty} \hdim E_{\bt_1,\bt_2}(\hv,v).
\end{equation*}
Thus, we only need to study the upper bound of the Hausdorff dimension of the set
\begin{equation*}
\left\{(x, y) \in[0,1]^2: \hat{v}_{\beta_1, \beta_2}(x, y) \geq \hat{v}\right\}
\end{equation*}
when $0<\hat{v} \leq 1$.

Write $v_0:=\frac{\hat{v}}{1-\hat{v}}$. For any sufficiently small $\delta>0$ and integer $l \geq 1$, we define
\begin{equation*}
F_{\beta_1,\beta_2}(\hat{v},l, \delta)=\left\{(x, y) \in[0,1]^2: \hat{v}_{\beta_1,\beta_2}(x, y) \geq \hat{v}, v_0+(l-1)\delta \leq v_{\beta_1,\beta_2}(x, y) \leq v_0+l\delta\right\}
\end{equation*}
and
\begin{equation*}
F_{\beta_1,\beta_2}(\hat{v},\infty)=\left\{(x, y) \in[0,1]^2: \hat{v}_{\beta_1,\beta_2}(x, y) \geq \hat{v}, v_{\beta_1,\beta_2}(x, y)=\infty \right\} .
\end{equation*}
Fix $\delta>0$. It follows from \eqref{tm1.6} that
\begin{equation*}
\left\{(x, y) \in[0,1]^2: \hat{v}_{\beta_1,\beta_2}(x, y)\geq\hat{v}\right\}
=\left( \bigcup_{l=1}^{+\infty} F_{\beta_1,\beta_2}(\hat{v}, l, \delta)\right)\bigcup F_{\beta_1,\beta_2}(\hat{v},\infty).
\end{equation*}
Hence,
\begin{equation*}
\hdim \left\{(x, y) \in[0,1]^2: \hat{v}_{\beta_1,\beta_2}(x, y) \geq \hat{v}\right\} =\max\{ \sup _{l\geq 1} \hdim  F_{\beta_1,\beta_2}(\hat{v}, l, \delta),~ \hdim F_{\beta_1,\beta_2}(\hat{v},\infty)\}.
\end{equation*}
When $v_{\beta_1,\beta_2}(x, y)=\infty$, notice that
\begin{equation*}
\left\{(x, y) \in[0,1]^2: \hat{v}_{\beta_1,\beta_2}(x, y) \geq \hat{v}, v_{\beta_1,\beta_2}(x, y)=\infty \right\} \subset \left\{(x, y) \in[0,1]^2: v_{\beta_1,\beta_2}(x, y)=\infty \right\}.
\end{equation*}
By Theorem \ref{thm1}, we know that $\hdim \left\{(x, y) \in[0,1]^2: v_{\beta_1,\beta_2}(x, y)=\infty \right\}=0$. It follows that
\begin{equation*}
\hdim F_{\beta_1,\beta_2}(\hat{v},\infty)=0.
\end{equation*}
Therefore, for any $\delta>0$,
\begin{equation}\label{401}
\hdim \left\{(x, y) \in[0,1]^2: \hat{v}_{\beta_1,\beta_2}(x, y) \geq \hat{v}\right\} =\sup _{l\geq 1} \hdim  F_{\beta_1,\beta_2}(\hat{v}, l, \delta).
\end{equation}
Next, we show that
\begin{equation*}
\hdim \left\{(x, y) \in[0,1]^2: \hat{v}_{\beta_1,\beta_2}(x, y)\geq \hat{v}\right\} \le \sup\limits_{0\le v \le \infty} \hdim E_{\bt_1,\bt_2}(\hv,v).
\end{equation*}

We divide the proof into a sequence of claims.

\medskip

{\sc Claim 1.} For any $ 0<\hv\le 1,$ $l>1$ and $\delta>0,$ we have \begin{align}\label{eqf}
\hdim  F_{\beta_1,\beta_2}(\hat{v}, l,\delta)\le&
\min\{ \tilde{A}(\hat{v},v_0+l\delta),~\tilde{B}(\hat{v},v_0+l\delta)\} \cdot 1_{\beta_1^{1+v_0+l\delta}>\beta_2}\\
&+\min \{\tilde{A}(\hat{v},v_0+l\delta), \tilde{C}(\hat{v},v_0+l\delta)\} \cdot 1_{\beta_1^{1+v_0+l\delta} \leq \beta_2}, \nonumber
\end{align}
where
\begin{align*}
\tilde{A}(\hat{v},v_0+l\delta)=&\left(1+\log _{\beta_2} \beta_1\right)\frac{v_0+(l+2)\delta-\hat{v}[1+v_0+(l+1)\delta]}{(1+v_0+l\delta)(v_0+l\delta-\hat{v})}+1-\log _{\beta_2} \beta_1,\\
\tilde{B}(\hat{v},v_0+l\delta)=&\left(1+\log _{\beta_1} \beta_2\right)\frac{v_0+(l+2)\delta-\hat{v}[1+v_0+(l+1)\delta]}{(1+v_0+l\delta)(v_0+l\delta-\hat{v})},
\\
\tilde{C}(\hat{v},v_0+l\delta)=&\frac{v_0+(l+2)\delta-\hat{v}-\hat{v}[v_0+(l+1)\delta](1+\log _{\beta_1} \beta_2)}{(1+v_0+l\delta)(v_0+l\delta-\hat{v})}+1.
\end{align*}

Equation \eqref{eqf} can be proved by employing a method analogous to that in Section \ref{section1}; here we only outline the differences.

\textbullet~ For any $\varepsilon>0$, we have
\begin{equation*}
\sum_{i=1}^{k-1}\left(m_i-n_i\right) \geq\left[\frac{(v_0+(l+1)\delta) \hat{v}}{v_0+(l+2)\delta-\hat{v}}-\varepsilon\right] n_k
\end{equation*}
for sufficiently large $k$. We also have
\begin{equation*}
\begin{aligned}
D_{n_1, m_1 ; \cdots ; n_{k-1}, m_{k-1}}=\left\{(\omega, \nu) \in \Sigma_{\beta_1}^{n_k} \times \Sigma_{\beta_2}^{n_k}: \omega_{n_i+1}\right. & =\cdots=\omega_{m_i}=0, \text { and } \\
\nu_{n_i+1} & \left.=\cdots=\nu_{m_i}=0, \text { for all } 1 \leq i \leq k-1\right\} .
\end{aligned}
\end{equation*}

\textbullet~
Replace $\Omega$ by
\begin{equation*}
\begin{aligned}
\Omega=\left\{(\{n_k\},\{m_k\}):\right. & \eqref{Ppstn}~\text{holds}, ~\liminf_{k\rightarrow\infty} \frac{m_k-n_k}{n_{k+1}}\geq \hat{v}
\\& \left. \text{~and}~v_0+(l-1)\delta\leq \limsup_{k\rightarrow\infty} \frac{m_k-n_k}{n_k}\leq v_0+l\delta \right\} .
\end{aligned}
\end{equation*}

\textbullet~
The estimate
\begin{equation*}
\sharp D_{n_1, m_1 ; \cdots ; n_{k-1}, m_{k-1}} \leq\left(\beta_1-1\right)^{-2 k} \beta_2^{k\left(1+\log _{\beta_2} \beta_1\right)} \beta_2^{n_k\left(1+\log _{\beta_2} \beta_1\right)\left[\frac{v_0+(l+2)\delta-\hat{v}(1+v_0+(l+1)\delta)}{v_0+l\delta-\hat{v}}+\varepsilon\right]}
\end{equation*}
holds.

\medskip
{\sc Claim 2.} When $0\leq v_{\beta_1,\beta_2}(x, y)<\infty$, we aim to prove that
\begin{equation}\label{402}
\sup _{l\geq1} \hdim  F_{\beta_1,\beta_2}(\hat{v}, l,\delta)
\leq  \sup _{l\geq1} \hdim E_{\beta_1, \beta_2}\left(\hat{v}, v_0+l \delta\right)+2(1+\log _{\beta_1} \beta_2)\delta.
\end{equation}

To this end, we note the following inequalities
\begin{equation*}
\tilde{A}(\hat{v},v_0+l\delta)-A(\hat{v},v_0+l\delta)=\frac{(1+\log _{\beta_2} \beta_1)\delta}{(1+v_0+l\delta)(v_0+l\delta-\hat{v})}\leq 2(1+\log _{\beta_2} \beta_1)\delta,
\end{equation*}
\begin{equation*}
\tilde{B}(\hat{v},v_0+l\delta)-B(\hat{v},v_0+l\delta)=\frac{(1+\log _{\beta_1} \beta_2)\delta}{(1+v_0+l\delta)(v_0+l\delta-\hat{v})}\leq 2(1+\log _{\beta_1} \beta_2)\delta,
\end{equation*}
and
\begin{equation*}
\tilde{C}(\hat{v},v_0+l\delta)-C(\hat{v},v_0+l\delta)=\frac{\delta}{(1+v_0+l\delta)(v_0+l\delta-\hat{v})}
\leq 2\delta.
\end{equation*}

\medskip
{\sc Claim 3.} \begin{equation}\label{403}
\lim _{\delta \rightarrow 0} \sup _{l \geq 1} \hdim E_{\beta_1, \beta_2}(\hat{v}, v_0+l\delta)=\sup _{0 \le v \le \infty} \hdim E_{\beta_1, \beta_2}(\hat{v}, v).
\end{equation}

By the definition of the supremum, for any $\eta>0$, there exists $v^* \geq v_0$ such that
\begin{equation*}
\hdim E_{\beta_1, \beta_2}(\hat{v}, v^*)>\sup _{ v \geq v_0} \hdim E_{\beta_1, \beta_2}(\hat{v}, v))-\eta.
\end{equation*}
Choose a sufficiently small $\delta$ satisfying $\delta<\eta$ and let $l=\lfloor \frac{v^*-v_0}{\delta} \rfloor$. Then $v_0+l\delta\in[v^*-\delta,v^*]$. Since $\hdim E_{\beta_1, \beta_2}(\hat{v}, v^*)$ is continuous with respect to $v$, we obtain
\begin{equation*}
\hdim E_{\beta_1, \beta_2}(\hat{v}, v_0+l\delta)\geq \hdim E_{\beta_1, \beta_2}(\hat{v}, v^*)-\eta>\sup _{ v \geq v_0} \hdim E_{\beta_1, \beta_2}(\hat{v}, v))-2\eta.
\end{equation*}
Hence,
\begin{equation*}
\sup _{l\geq1} \hdim E_{\beta_1, \beta_2}(\hat{v}, v_0+l\delta)>\sup _{ v \geq v_0} \hdim E_{\beta_1, \beta_2}(\hat{v}, v))-2\eta.
\end{equation*}
By Theorem \ref{thm3}, $\hdim E_{\beta_1, \beta_2}(\hat{v}, v)=0$ for all $0\leq v<v_0$ , so
\begin{equation*}
\sup _{0 \leq v \leq \infty} \hdim E_{\beta_1, \beta_2}(\hat{v}, v)=\sup _{ v \geq v_0} \hdim E_{\beta_1, \beta_2}(\hat{v}, v).
\end{equation*}
Letting $\delta \rightarrow 0$(and correspondingly, $\eta\rightarrow 0$), we combine the above result with the trivial bound $\sup _{l\geq1} \hdim E_{\beta_1, \beta_2}(\hat{v}, v_0+l\delta)\leq \sup _{0 \leq v \leq \infty} \hdim E_{\beta_1, \beta_2}(\hat{v}, v)$ to get \eqref{403}.

Finally, combining \eqref{401}, \eqref{402} and \eqref{403}, and letting $\delta \rightarrow 0$, we conclude that
\begin{align*}
\hdim \left\{(x, y) \in[0,1]^2: \hat{v}_{\beta_1,\beta_2}(x, y)\ge\hv\right\}
&\le \lim_{\delta \to 0}\sup _{l \geq 1} \hdim F_{\beta_1, \beta_2}(\hat{v}, l, \delta) \\
&\leq \lim_{\delta \to 0}\left(\sup _{l \geq 1} \hdim E_{\beta_1, \beta_2}(\hat{v}, v_0+l\delta)+(1+\log _{\beta_1} \beta_2)\delta\right)\\
&= \sup _{0 \leq v \leq \infty} \hdim E_{\beta_1, \beta_2}(\hat{v}, v).
\end{align*}This completes the proof.

\section*{Acknowledgment}
The authors would like to express their sincere gratitude to Professor Baowei Wang (Huazhong University of Science and Technology) for his helpful suggestions, and to Professor Jian Xu (Huazhong University of Science and Technology) for his valuable guidance and insightful comments. They also wish to thank Dr. Jing Feng (Wuhan University of Science and Technology) for her continuous support and helpful feedback. This work is supported by National Key R\&D Program of China (No. 2024YFA1013700) and Natural Science Foundation of China (No. 12331005).

\renewcommand{\refname}{
\end{document}